%% file: 00main.tex
\documentclass[leqno,12pt]{amsart}
\pdfoutput=1
\usepackage{amsmath,amsthm,amssymb,amscd,stmaryrd, mathrsfs}
\usepackage[margin=1.5in]{geometry}
\usepackage{ascmac}

\usepackage[mathscr]{eucal}
\usepackage[all,cmtip]{xy}
\usepackage{tikz-cd}
\usetikzlibrary{decorations.pathmorphing}
\usetikzlibrary{decorations.markings}
\usepackage{enumerate}
\usepackage{graphicx}
\usepackage{xcolor}
\usepackage{microtype}
\usepackage{enumitem}
\usepackage{verbatim}
\usepackage{version}
\usepackage{color}
\usepackage{float}
\usepackage{mathtools}
\usepackage[font=small,skip=0pt]{caption}
\usepackage{hyperref}
\usepackage{nicematrix}

\usepackage{hhline}

\usepackage{blindtext}
\newcommand{\CC}{{\mathbb C}}
\newcommand{\ZZ}{{\mathbb Z}}

\newcommand{\NN}{\mathbb N}

\newcommand{\AAA}{{\mathbb A}}

\newcommand{\GL}{\operatorname{GL}}

\newcommand{\Hom}{\operatorname{Hom}}
\newcommand{\Lie}{\operatorname{Lie}}
\newcommand{\Ker}{\operatorname{Ker}}

\newcommand{\Ima}{\operatorname{Im}}

\newcommand{\GIT}{/\!\!/}
\newcommand{\II}{\mathcal{I}}
\newcommand{\EE}{\mathcal{E}}
\newcommand{\GG}{\mathcal{G}}

\newcommand{\IS}{\mathcal{I}^s}
\newcommand{\M}{\mathcal{M}}
\newcommand{\TM}{\widetilde{{\mathcal{M}}}}

\newcommand{\Liegl}{\mathfrak{gl}}
\newcommand{\Liesl}{\mathfrak{sl}}
\newcommand{\aD}{\prescript{a}{}{D}}
\newcommand{\ad}{\prescript{a}{}{d}}
\newcommand{\cd}{\prescript{c}{}{d}}
\newcommand{\cD}{\prescript{c}{}{D}}
\newcommand{\Y}{\mathscr{Y}}
\newcommand{\vv}{\underline{v}}

\numberwithin{equation}{section}
\newtheorem{Theorem}{Theorem}[section]
\newtheorem{Corollary}[Theorem]{Corollary}
\newtheorem{Lemma}[Theorem]{Lemma}
\newtheorem{Proposition}[Theorem]{Proposition}
\newtheorem*{Theorem*}{Theorem}

\theoremstyle{definition}
\newtheorem{Definition}[Theorem]{Definition}
\newtheorem{Example}[equation]{Example}

\theoremstyle{remark}
\newtheorem{Remark}[Theorem]{Remark}

\title{Supersymmetry for brane diagrams and bow varieties}
\author{Tiziano Gaibisso}
\address{Department of Mathematics, Imperial College London, London, UK}
\email{t.gaibisso22@imperial.ac.uk}

\begin{document}
\begin{abstract}
We provide combinatorial and numerical criteria to characterize affine type A bow diagrams giving rise to a non-empty bow variety. The key idea is to prove that such diagrams correspond to supersymmetric brane systems in type IIB string theory, allowing us to reformulate the problem in purely combinatorial terms. To achieve this, we characterize supersymmetry for affine type A brane systems (and, by extension, for types B, C, and D) using Hanany–Witten transitions. This leads to a finite-step algorithm that decides whether a given affine type A bow or brane diagram is supersymmetric, which consists in checking a finite set of inequalities, so providing a numerical criterion for non-emptiness. Finally, we provide a different perspective by introducing a further criterion in terms of weights of affine Lie algebras.

Along the way, we also prove that increasing dimension vectors between two consecutive x-points or arrows in a bow diagram (not necessarily of type A) preserves the properties of generating non-empty bow varieties.
\end{abstract}
\maketitle

\tableofcontents

\input{001Intro}
\addtocontents{toc}{\protect\setcounter{tocdepth}{1}}
\subsection*{Acknowledgement}
I am deeply grateful to Travis Schedler for his guidance and support throughout this work. I would like to thank Amihay Hanany for his invaluable explanations, and Hiraku Nakajima for insightful discussions and for bringing to my attention the relation between weights of affine Lie algebras and the stratification of bow varieties. I also thank Sam Bennet, Guhesh Kumaran, and Chunhao Li for their feedback on earlier versions of this paper. This work forms part of the author's PhD project at Imperial College London.

\addtocontents{toc}{\protect\setcounter{tocdepth}{1}}
\input{002Brane}
\input{002CriterionForStratum}
\input{003Supersymmetryandbowvarieties}
\input{004supersymmetryinequalitiesandthestratumcondition}

\input{004Algorithm}

\input{005example}

\input{00Appendix2}
\input{00Appendix}

\bibliographystyle{alpha}
\bibliography{00main}

\end{document}

%% file: 001Intro.tex
\section{Introduction}
\subsection{Motivation and key ideas}
Bow varieties were introduced by Cherkis \cite{cherkis2011instantons} as an ADHM-type description of certain moduli spaces of instantons on the Taub-NUT space. They provide a natural generalization of Nakajima quiver varieties \cite{Nak94} \cite{nakajima1998quiver} that has attracted considerable attention. The algebro-geometric study of bow varieties was initiated by Nakajima and Takayama \cite{Tak16}, \cite{NT17}, who developed a `quiver description' for those of affine type A (see also \cite{gaibisso2024quiver} for a generalization to arbitrary type). This approach realizes bow varieties as moduli spaces of certain quiver representations, offering a powerful framework to address algebro-geometric questions and explore their connection with representation theory (see also \cite{nakajima2018towards}).

Furthermore, every construction of a bow variety begins with combinatorial data known as \textit{bow diagrams}, which generalize the data in the Nakajima's varieties given by a quiver together with a framed dimension vector. In affine type A, there is a known relationship between bow diagrams and brane systems in type IIB string theory \cite{C09, SCh11super}. 

This perspective on bow
diagrams was further elucidated in \cite{NT17}. For example, they proved that Hanany–Witten transitions for 5-branes induce an equivalence relation on bow diagrams (see \S\ref{s.hw}). Additionally, they also proved that Coulomb branches of affine type A supersymmetric quiver gauge theories \cite{SW94, BDG17, Nak15a, nakajima2015towards, braverman2016towards, BFNI} are bow varieties. Therefore, affine type A bow varieties provide a natural framework to observe symplectic duality \cite{BLPW16} and $3d$ mirror symmetry or S-duality (see, for example, \cite{RS20}). For an expository work on $3d$ mirror symmetry, see, for instance, \cite{WY2023}.

The present work was motivated by the basic question: when does a bow diagram generate a non-empty bow variety? For Nakajima quiver varieties, this is trivial, as any quiver with any framed dimension vector generates a non-empty variety (for instance, the one with trivial deformation and stability parameters). However, it is not difficult to construct bow diagrams that generate only empty bow varieties. The simplest example is given by:
\begin{equation}\label{introexample}
    \input{Fig/Introexample}
\end{equation}

See also Example \ref{ex1}. More generally, any diagram which admits a sequence of Hanany—Witten transitions to a diagram having a negative dimension generates only empty bow varieties. It turns out the converse, while not obvious, is true, and is proved in our main theorem (see Section \ref{s.step1}). In string theory, there is a further notion of brane diagrams which ``break supersymmetry" \cite{HW}. The example ~\eqref{introexample} clearly is one such diagram, and another part of our main theorem is that this condition is also equivalent to generating only empty bow varieties.

In this paper, we characterize supersymmetry in affine type A brane systems using Hanany--Witten transitions, deriving a numerical criterion that involves infinitely many inequalities.
In the process we find an algorithm for testing the supersymmetry condition, and prove that for any given configuration, it suffices to test only finitely many of the inequalities to determine whether it is supersymmetric.
To our knowledge, no such characterization existed in literature. It follows that supersymmetry is equivalent to generating a non-empty bow variety, and that, in such a case, the variety with trivial deformation and stability parameters is always non-empty. Finally, we use Nakajima--Takayama’s stratification of bow varieties \cite[Theorem 7.26, Remark 7.27]{NT17} to provide an additional characterization of supersymmetry in terms of dominant weights of affine Lie algebras.

\subsection{Overview}
Let us now make the considerations above more precise.  

In this paper, an affine type A \textit{brane diagram} $\mathcal{B}$ consists of:  
\begin{enumerate}
    \item A graph with two sets of vertices, denoted by $\vert$ (NS5-branes) and $\times$ (D5-branes), connected by edges (D3-branes).
    \item Immersions of the edges together with an embedding of the set of vertices into a fixed circle, up to reparametrisation.
\end{enumerate}
See Definition \ref{Definitionbranediagram} and Figure \ref{fig:affinesystems}.

Given a brane diagram $\mathcal{B}$, we construct a \textit{bow diagram} as follows (Definition \ref{definitionbowdiagram}). First, we replace the NS5-branes with anticlockwise-oriented arrows connected by anticlockwise-oriented wavy lines going from the endpoint of an arrow to the starting point of the next arrow.  This graph is called \textit{bow}. Next, we replace the D5-branes with points, called \textit{x-points}, placed on the wavy lines. The collection of x-points and arrows partitions the wavy lines into segments representing the spaces between 5-branes in the brane diagram. Finally, we assign to each segment $\zeta$ the total number $v_\zeta$ of segments of D3-branes which run through the corresponding space. The vector $\vv=(v_\zeta)_\zeta$ is called \textit{dimension vector}. In this work, we extend the notion of a bow diagram by allowing the dimension vector $\vv$ to have negative integer entries.

Thus, a \textit{bow diagram} is a triple $(B, \Lambda, \vv)$ where $B$ denotes the bow, $\Lambda$ the set of x-points and $\vv$ the dimension vector. See Figures \ref{fig:branetobow} and \ref{fig:4}.

Analogously to the Nakajima quiver varieties, a bow diagram $(B,\Lambda,\vv)$ with nonnegative dimensions (i.e. originating from a brane diagram), determines a symplectic affine variety $\TM$ carrying a Hamiltonian action by a reductive group $\GG$ with moment map denoted by $\mu$. Then, one defines the \textit{bow variety} associated to $(B,\Lambda,\vv)$, with deformation parameter $\lambda\in(\operatorname{Lie}(\GG)^*)^\GG$ and stability parameter $\theta$, where $\theta$ is a rational $\GG$-character, as the Hamiltonian reduction $\M_{\lambda,\theta}=\mu^{-1}(\lambda)\GIT_\theta\GG$ (see \S\ref{s.bowvarieties}). If some dimension $v_\zeta$ is negative, then we define $\M_{\lambda,\theta}=\varnothing$.

Let us now recall the notion of supersymmetry (see \S\ref{s.branesystemsandbowdiagrams}). In physics, D3-branes with endpoints on 5-branes of different types are called \textit{fixed} (because, unlike those with endpoints on the same type, they have no freedom to move in the ambient space containing all the branes). See Figure \ref{fig:finitesystem} and Definition \ref{definitionfixed}. In this paper, we adopt the convention that a fixed D3-brane is oriented by regarding its endpoint on the NS5-brane as the starting point. An affine type A brane diagram is said to be \textit{supersymmetric} if, for each pair of 5-branes of different types and every $t \in \mathbb{N}$, there is at most one clockwise-oriented fixed D3-brane and at most one anticlockwise-oriented fixed D3-brane connecting them, each completing $t$ full loops. A bow diagram is supersymmetric if it originates from a supersymmetric brane diagram. 

In string theory, the Hanany--Witten transition \cite{HW} describes a local process of brane creation and annihilation, inducing an equivalence relation on brane diagrams. In terms of bow diagrams, it takes the following form:
\begin{equation}
    \begin{aligned}\input{Fig/HWdiagramIntro}\label{hwtransition-intro}\end{aligned}
\end{equation}
Nakajima and Takayama \cite[Proposition 7.1]{NT17} translated this process into isomorphisms of bow varieties, which means that two bow diagrams related by a sequence of Hanany--Witten transitions generate isomorphic bow varieties (see Theorem \ref{HWtransitionsforbow}).

Given a bow diagram, by successive application of Hanany--Witten transitions, we can always separate x-points and arrows by moving all the x-points on the same wavy line (see Figures \ref{StypeAffine} and \ref{StypeA}). For a separated affine type A bow diagram, the fact that every sequence of Hanany--Witten transitions does not generate negative dimensions can be translated into a set of infinitely many inequalities, which will be called \textit{supersymmetry inequalities} (see Definition \ref{HWinequalities}). 

Let us define a bow subdiagram of $(B,\Lambda,\vv)$ as a bow diagram $(B,\Lambda,\vv')$ with $v'_\zeta\leq v_\zeta$ for every segment $\zeta$. Then, we say that a bow diagram satisfies the \textit{stratum condition} (Definition \ref{def:stratumcondition}) if every separated Hanany--Witten equivalent bow diagram admits a subdiagram satisfying the following properties:
\begin{itemize}
    \item for every x-point, the difference between the adjacent dimensions is unchanged;
    \item we can apply a sequence of Hanany--Witten transition to obtain a bow diagram with nonnegative dimensions and satisfying the balanced condition (that is a diagram such that for every arrow, the difference between the adjacent dimensions is zero).
\end{itemize}
Due to the correspondence between balanced bow diagrams and weights of affine Lie algebras (see \S\ref{s.balance}), this condition can be reformulated as a numerical criterion in terms of weights of affine type A Kac--Moody Lie algebra (see \S\ref{s.stratumcriterion}). This criterion is equivalent to requiring the existence of a non-empty stratum of the non-deformed bow variety $\M_{0,0}$ and can, in fact, be derived using the same arguments as in \cite[Theorem 7.26, Remark 7.27]{NT17} (see Remark \ref{Rem:stratumconditionexistencestratum}). However, in this work, we will not rely on this interpretation.

We are ready to state the main result of the paper.
\begin{Theorem*}[See Theorem \ref{maintheorem}]
    Given a bow diagram $(B,\Lambda,v)$ of affine type A, the following are equivalent.
\begin{enumerate}[label=(\alph*)]
\item $\M_{0,0} \neq \varnothing$.
    \item There exists $(\lambda,\theta) \in \CC^\II\times\ZZ^\II$, such that $\M_{\lambda,\theta}\neq\varnothing$.
    \item Every Hanany--Witten equivalent bow diagram originates from a brane diagram.
    \item Any Hanany--Witten equivalent separated bow diagram satisfies supersymmetry inequalities.
    \item It is supersymmetric.
    \item It satisfies the stratum condition.
    \end{enumerate}
\end{Theorem*}
The proof is divided into five main steps, which are summarized in the following diagram.
\begin{equation*}
    \input{Fig/Summaryproof}
\end{equation*}
The equivalence $(c)\Leftrightarrow(d)\Leftrightarrow(e)$ is nothing but a characterization of supersymmetry for bow diagrams in type A and it does not involve bow varieties. 

Supersymmetry is preserved under S-duality for affine type A brane systems, which consists in replacing every NS5-brane with a D5-brane and vice versa \cite[\S II.D]{GK99}, \cite[\S 7.1]{NT17}. In Remark \ref{remarkSdualstratumcondition} we also discuss S-dual counterparts of the equivalent conditions in the main theorem.

Let us briefly discuss the proof of the main theorem. It is clear that $(a)\Rightarrow(b)$. The implication $(e)\Rightarrow(c)$ follows from Hanany--Witten's work \cite{HW}. $Step\ 1$ is straightforward and is also a consequence of \cite[Corollary 3.14]{SW23} and \cite[Proposition 7.1]{NT17}. 

$Step\ 2$ is also straightforward and is the motivation for the supersymmetry inequalities. A direct proof of the reverse implication, i.e. $(c)\Leftarrow(d)$, requires a more involved argument. Although this is not essential for the formal proof of the main theorem, we will provide a proof of it in Appendix \ref{InverseStep2-Appendic}, as the argument is interesting in its own right and, as far as we know, introduces new methods for carrying out computations with Hanany--Witten transitions in affine type A bow diagrams.

Finally, $Steps\ 3$, $4$ and $5$ are the main steps. Their proof strongly relies on the idea of building upon finite type A diagrams by using two moves: Hanany--Witten transitions and \textit{supersymmetric increments} (i.e. increments of the dimension vector of a bow diagram corresponding to adding some D3-branes with endpoints on 5-branes of the same type). See Definition \ref{Def-susyincrements}.
\begin{Theorem*}[see Theorem \ref{supersymmetricincrementsforbows}, Corollary \ref{corollarysupersymmetricincrementsforbow}]
If a bow diagram generates a non-empty bow variety, then every bow diagram obtained via supersymmetric increments generates a non-empty bow variety.
\end{Theorem*}
We notice that such increments can be generalized to the case of bow diagrams with any underlying quiver \cite{gaibisso2024quiver}. See Remark \ref{Rem.susyincrementsforbowdiagramsingeneraltype}.

From the proof of the main theorem, we derive a finite-step algorithm for detecting supersymmetry in bow diagrams. Additionally, we present a method to construct both a supersymmetric brane diagram and an element of $\M_{0,0}$ associated with supersymmetric bow diagrams. In Section \ref{section4}, we provide a detailed discussion of the algorithm and these methods, without requiring the reader to be familiar with the proof of the main theorem.

As noticed in Remark \ref{Remark.b.c.d.}, although we only deal with type A brane diagrams, our algorithm also provides a systematic method to detect supersymmetry for brane diagrams of type B,C and D (see for example \cite{bourget2023branes}).

The paper is organized as follows. In Section \ref{section2}, we start recalling the definition of brane and bow diagrams. Then we introduce the notion of supersymmetric diagrams and supersymmetric increments. Finally, we review the quiver description of bow varieties. In Section \ref{SectionHW}, we recall Hanany–Witten transitions for bow varieties and introduce supersymmetry inequalities. In Section \ref{SectionWandStratum}, we recall the notion of balanced bow diagrams and their relation with weights of affine Kac--Moody Lie algebras. Then we introduce the stratum condition. In Section \ref{section3} we state and prove the main theorem.  In section \ref{section4} we present the algorithm to detect supersymmetry of bow diagrams. We also present a systematic method to construct supersymmetric brane diagrams and elements of $\M_{0,0}$ associated with supersymmetric bow diagrams. In Section \ref{section5} we provide explicit examples to illustrate how the algorithm works. In Appendix \ref{InverseStep2-Appendic} we provide a direct proof of $(d)\Longrightarrow(c)$. In Appendix \ref{s.weights}, we review useful facts about weights of affine Lie algebras, generalized Young diagrams and level-rank duality.

%% file: 002Brane.tex
\section{Brane and bow diagrams, supersymmetry and bow varieties}\label{section2}

\subsection{Brane and bow diagrams}\label{s.branesystemsandbowdiagrams}
In this subsection, we introduce affine type A brane diagrams and bow diagrams.

We begin by recalling the notion of a brane system of finite type A, which will motivate some of the definitions introduced later \cite{HW}. A \textit{brane} is an affine subspace of a ten-dimensional real vector space. We are interested in three types of branes, denoted, respectively, by NS5, D3 and D5, of dimensions, respectively, $6,4$ and $6$. Precisely, we fix a system of coordinates, say $(x_0,\dots,x_9)$. Then, NS5-branes are $6$-dimensional spaces with fixed values of $x_6,\dots,x_9$, that is, they run along $(x_0,x_1,x_2,x_3,x_4,x_5)$. Similarly, D3-branes run along $(x_0,x_1,x_2,x_6)$ and D5-branes run along $(x_0,x_1,x_2,x_7,x_8,x_9)$. Furthermore, D3-branes are stretched between 5-branes (possibly of different types). See Figure \ref{fig:finitesystem}. A \textit{brane system} or \textit{brane configuration} of finite type A is a collection of NS5, D3 and D5 branes. 
So, in the sense of Figure \ref{fig:finitesystem}, we will deal with brane systems like a configuration of lines in $\AAA^3$.
\begin{figure}[H]
    \centering
    \input{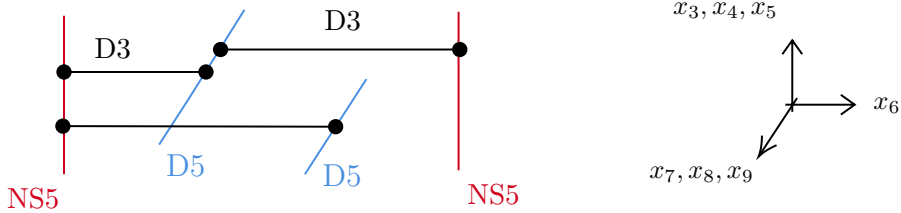}
    \caption{Vertical segments are NS5-branes, diagonal segments are D5-branes and horizontal segments are D3-branes. Black dots denote extreme points of D3-branes.}
    \label{fig:finitesystem}
\end{figure}

\begin{Definition}\label{Definitionbranediagram}
    A \textit{brane diagram}, denoted by $\mathcal{B}$, is given by the following data.
    \begin{enumerate}
\item     A graph defined by:
\begin{itemize}
    \item two types of vertices given by either $\vert$ or $\times$ and called, respectively, NS5-branes and D5-branes;
    \item edges connecting vertices (independently from the types).
\end{itemize}
\item Immersions of the edges together with an embedding of the discrete set of all vertices into a fixed segment or circle, up to reparametrisation. 
\end{enumerate}
A brane diagram is said to be of affine type A if the immersions are taken into a circle, and of finite type A if they are taken into a segment. See Figure \ref{fig:affinesystems}
\end{Definition}

\begin{figure}[H]
\label{branesystem}
    \centering
    \input{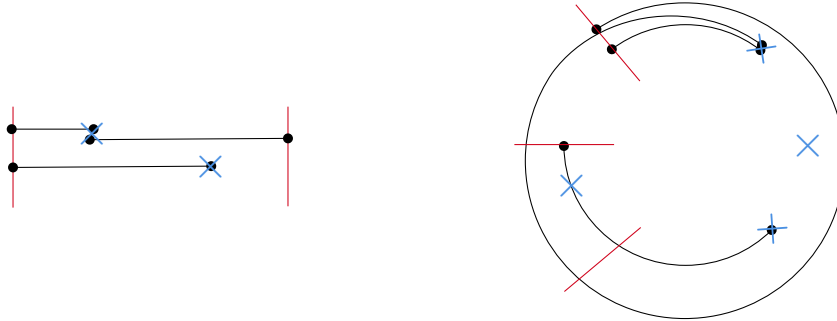}
    \caption{A fine type A brane diagram on the left and an affine type A brane diagram on the right.}
    \label{fig:affinesystems}
\end{figure}
In particular, there exists a correspondence between brane system of finite type A and brane diagrams immersed in a segment by replacing NS5-branes with $\vert$, D5-branes with $\times$ and D3-branes with edges. Although affine type A brane systems can be constructed similarly, for our purposes it will be sufficient to work solely with brane diagrams.

Given a brane diagram of affine type A we can construct a further diagram, called \textit{bow diagram}, as follows. See \cite{cherkis2011instantons} for the original definition.

We take an anticlockwise oriented quiver of affine type $A_{n-1}$ where $n$ is the number of NS5-branes.  Each vertex of the quiver is then replaced by anticlockwise-oriented \textit{wavy line} \input{Fig/wavyline}. Specifically, at each vertex, a wavy line is drawn starting from the endpoint of the incoming arrow and ending at the starting point of the outgoing arrow. Such a graph is called \textit{bow} and denoted by $B=(\II,\EE)$ where $\II$ is the set of wavy lines and $\EE$ is the set of arrows.

Next, since wavy lines encode the space between NS5-branes, it makes sense to replace vertices $\times$ (i.e. D5-branes) with points on wavy lines, still denoted by $\times$, called \textit{x-points}. Let us notice that this is possible because of the existence of the embedding of vertices in the definition of brane diagrams (Definition \ref{Definitionbranediagram}). 
By construction, the collection of x-points partitions the wavy lines into parts called \textit{segments}, which correspond to the spaces between 5-branes (not necessarily of the same type) in the brane diagram. We denote the set of segments by $\IS$.

Finally, we assign to each segment $\zeta$ the total number $v_\zeta$ of segments of D3-branes which run through the corresponding space. See Figure \ref{fig:branetobow}.
\begin{figure}
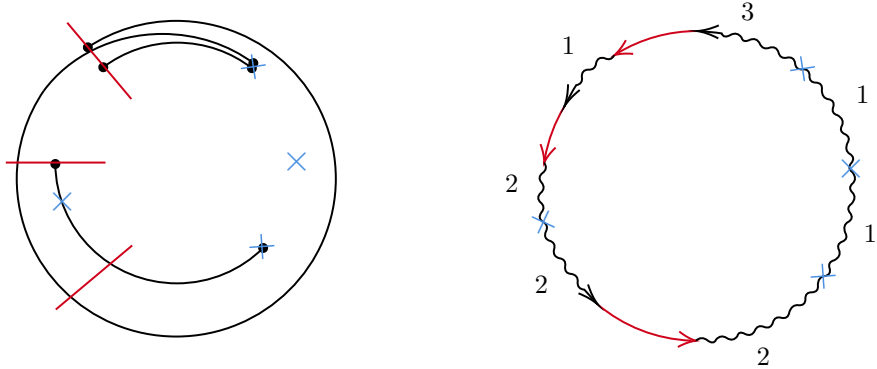

    \centering
    \include{Fig/Branetobow}
    \caption{Bow diagram associated with a brane diagram.}
    \label{fig:branetobow}
\end{figure}

In this work, we extend the notion of a bow diagram by allowing the dimension vector $\vv$ to have negative entries.

\begin{Definition}\label{definitionbowdiagram}
   A \textit{bow diagram} is a triple $(B, \Lambda, \vv)$, where $B$ is a \textit{bow}, $\Lambda$ is a set of x-points on wavy lines, and $\vv \in \ZZ^{\IS}$ is a vector called the \textit{dimension vector}. If $\vv\in\NN^{\IS}$, we say that the bow diagram \textit{originates from a brane diagram}.
\end{Definition}
It is often easier to work with \textit{simplified} bow diagrams of affine type A, which are obtained by replacing the arrows with $\bigcirc$ and removing waves from the wavy lines. See Figure \ref{fig:4}.
\begin{figure}
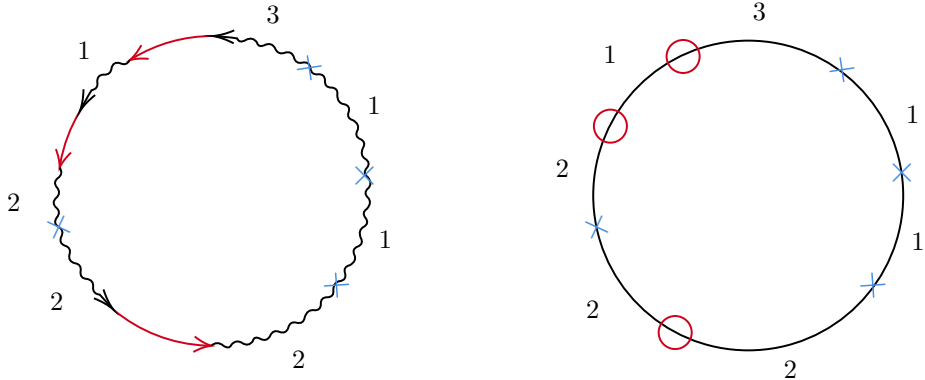

    \centering
    \include{Fig/Figure-Bowdiagrams}
    \caption{Bow diagram of affine type A on the left-hand side and its simplified version on the right-hand side.}
    \label{fig:4}
\end{figure}

\begin{Remark}
    Note that two distinct brane diagrams can give rise to the same bow diagram. We can think of bow diagrams originating from brane diagrams as brane diagrams for which we do not know where the D3-branes end.
\end{Remark}
\begin{Definition}\label{Definitonsubdiagram}
    A bow subdiagram of $(B,\Lambda,\vv)$ is a triple $(B,\Lambda,\vv')$ with $v'_\zeta\leq v_\zeta$ for every segment $\zeta$.
\end{Definition}
Finite type A brane diagrams can be viewed as affine type A diagrams containing a pair of 5-branes that delimit a region of space through which no D3-brane passes. Similarly, finite type A bow diagrams can be regarded as affine type A bow diagrams $(B,\Lambda,\vv)$ with a segment $\zeta\in\IS$ such that $v_\zeta=0$. In this configuration, we can imagine ``cutting" the circle at $\zeta$ and ``opening" it to form a linear arrangement of branes, as, for example, in Figure \ref{fig:affinetofinite}.
\begin{figure}[H]
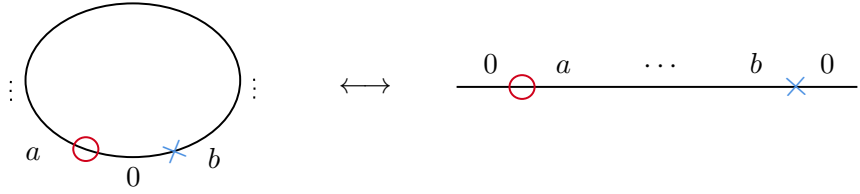

    \centering
    \include{Fig/affinetofinite}
    \caption{Equivalence between affine and finite type A bow diagrams.}
    \label{fig:affinetofinite}
\end{figure}
We notice that, a priori, there is not a unique way to open the circle into a line since there may be different segments with dimension $0$ that one could choose to cut the circle. However, for our purposes, all the possible choices will be equivalent.
\begin{Remark}
By construction, finite type A bow diagrams must always have extreme segments with dimension $0$. However, we allow extreme segments to have non-zero dimensions. Specifically, given a finite type A bow diagram, we can add an arrow connecting an extreme segment to a new wavy line with no $x$-points and dimension $0$. It is important to note that both the new arrow and the new wavy line have only one possible orientation, and the resulting diagram remains of finite type A. For our purposes, there is no difference between the original and the modified bow diagram.
\end{Remark}
\begin{Remark}
    The notations and definitions introduced earlier are not consistently used in the literature. For example, in the physics literature, what we call a ``brane system" is sometimes referred to as a ``brane diagram," while what we call a ``bow diagram" is often referred to as a ``brane diagram." Moreover, since the interpretation of simplified bow diagrams in terms of brane systems, it would seem natural to use the symbol $\vert$ instead of $\bigcirc$. However, we have opted to follow the notation in \cite{NT17}, as this choice helps to clearly distinguish between brane diagrams and bow diagrams.  Lastly, in some works on bow varieties of finite type A, bow diagrams are represented differently, replacing $\bigcirc$ with a segment of positive slope $/$ and $\times$ with a segment of negative slope $\backslash$ (see, for example, \cite{SW23}).
\end{Remark}

\subsection{Supersymmetry}
A particularly important class of brane diagrams is formed by the so-called supersymmetric brane diagrams. To define them, we need the following preliminary definitions.

The interpretation of brane diagrams as configurations of rigid $1$-dimensional affine subspaces in $\AAA^3$ (see Figure \ref{fig:finitesystem}) motivates the following definition.
\begin{Definition}\label{definitionfixed}
    Given an affine type A brane diagram, a D3-brane is said to be \textit{fixed} if it has endpoints on 5-branes of different types and \textit{unfixed} otherwise.
\end{Definition}
\begin{Definition}The \textit{winding number} of a fixed D3-brane in an affine type A brane diagram is the total number of full loops it completes.\end{Definition}
For example, in the affine type A brane diagram in Figure \ref{fig:affinesystems}, there are two D3-branes with winding number $0$ and one D3-brane with winding number $1$. Moreover, we notice that every fixed D3-brane in a finite type A diagram has winding number $0$.

In this paper, we adopt the convention that a fixed D3-brane is oriented by taking its endpoint on the NS5-brane as starting point. Therefore, it makes sense to talk about clockwise-oriented and anticlockwise-oriented fixed D3-branes in a brane diagram.
\begin{Definition}\label{definitionsupersymmetry}
A brane diagram of affine type A is \textit{supersymmetric} if, for every $t \in \mathbb{N}$ and every pair of 5-branes of different types, there is at most one clockwise-oriented D3-brane and at most one anticlockwise-oriented D3-brane connecting them, each with winding number $t$. A  bow diagram of affine type A is \textit{supersymmetric} if it originates from a supersymmetric brane diagram.
\end{Definition}
Let us illustrate the notion of supersymmetry for bow diagrams with an example. Consider a D5-brane and an NS5-brane in an affine type A brane diagram, connected by two anticlockwise oriented D3-branes and one clockwise oriented D3-brane, all having the same winding number $t$. By definition, such a configuration is not supersymmetric. Now, suppose these are the only branes \textit{breaking supersymmetry}. In this case, we can immediately conclude that the corresponding bow diagram is supersymmetric. Indeed, if we “glue” the clockwise D3-brane with one of the two anticlockwise D3-branes, we obtain a new brane diagram that is supersymmetric and which gives rise to the same bow diagram as the original brane diagram.

\begin{Definition}\label{Def-susyincrements}
Given a brane diagram, we define a \textit{supersymmetric increment} between NS5-branes (respectively, D5-branes) to be the operation of adding unfixed D3-branes whose endpoints are placed exclusively on NS5-branes (respectively, D5-branes).\end{Definition}

The term supersymmetric increment will refer to successive applications of the operations defined above.
\begin{Definition}
A bow diagram $(B,\Lambda,\vv')$ is said to be obtained from a bow diagram $(B,\Lambda,\vv)$ via a \textit{supersymmetric increment} between a pair of arrows $e,e'$ (respectively, x-points $x,x'$) if, for every segment $\zeta$, the difference $v'_\zeta - v_\zeta$ is non-negative and depends only on the arc of the bow diagram determined by $e,e'$ (respectively, by $x,x'$) that contains the segment $\zeta$.

The term supersymmetric increment will refer to successive applications of the operations defined above.
\end{Definition}

\begin{Remark}\label{Rem-susyincrements}
Let us consider a bow diagram $(B,\Lambda,\vv)$ which originates from a brane diagrams $\mathcal{B}$. Then, supersymmetric increments on $(B,\Lambda,\vv)$ correspond to supersymmetric increments on $\mathcal{B}$. Since supersymmetric increments on brane diagrams preserve supersymmetry, then also supersymmetric increments on bow diagrams preserve supersymmetry. However, although removing unfixed D3-branes from brane diagrams preserves supersymmetry, in general, we do not know how to decrease the dimension vector of a bow diagram while preserving supersymmetry. This is because the endpoints of D3-branes are not visible in bow diagrams. For this reason, we need to impose further assumptions on our bow diagrams to decrease the dimension vector in a way that is compatible with supersymmetry. See Proposition \ref{Prop.reducing} and Remark \ref{rem:susydecrements}.
\end{Remark}

\subsection{Bow varieties}\label{s.bowvarieties}
We recall the quiver description of bow varieties of affine type A introduced in \cite{NT17}.

Let us fix a bow diagram $(B,\Lambda,\vv)$ and suppose it originates from a brane diagram, i.e. $\vv\in\NN^{\IS}$. To define bow varieties we need to introduce two building blocks.
The first one is associated to every arrow of the bow. Given an arrow $e\in\EE$ we denote by $\zeta_{t(e)}$ and $\zeta_{h(e)}$ the segments, respectively, at the origin and at the end of the arrow, that is:
\begin{equation*}
    \input{Fig/arrow}
\end{equation*} We denote by $v_{t(e)}$ and $v_{h(e)}$ the corresponding entries of the dimension vector. Then, we associate to $e$ the symplectic vector space
\begin{equation*}
    \M^e=\operatorname{T}^*\Hom(\CC^{v_{t(e)}},\CC^{v_{h(e)}})=\Hom(\CC^{v_{t(e)}},\CC^{v_{h(e)}})\oplus \Hom(\CC^{v_{h(e)}},\CC^{v_{t(e)}})
\end{equation*}
whose elements will be denoted by $(C_e,D_e)$ and the second identification is given by the trace pairing.

The second space we need is associated to every x-point of the bow. Let us fix $x \in \Lambda$. We denote by $\zeta_x^-$ and $\zeta_x^+$ the adjacent segments to $x$ following the wavy line's orientation, that is:
 \begin{equation*}
     \input{Fig/AdjSeg}.
 \end{equation*}
For simplicity, we also use the notation $v_{\zeta_x^\pm}=v_x^\pm$. Then, a \textit{triangle at x} is defined as a collection of linear maps $(A_x,B_x^-,B_x^+,a_x,b_x)$ as follows:
\begin{equation*}
    \input{Fig/GenTri}.
\end{equation*}
Finally, we define the variety $\M^x$ as the collection of triangles at $x$ satisfying the following conditions:
    \begin{equation*}\label{condition-a}
        B_x^+A_x-A_xB_x^-+a_xb_x= 0. \tag{$a$}
    \end{equation*}
        \begin{equation*}\label{condition-S1}
        \text{There is no subspace }0\neq S \subset \Ker A_x \cap \Ker b_x,\text{ }B_x^-\text{-invariant.} \tag{S1}
        \end{equation*}
        \begin{equation*}\label{condition-S2}
        \text{There is no subspace }\Ima A+\Ima a \subset T \subsetneq V_{\zeta^{+}},\text{ }B_x^+\text{-invariant.}\tag{S2}
    \end{equation*}
The space $\M^x$ was introduced in \cite{Tak16}, where it is shown that it  is a symplectic affine variety.

Then, given a bow diagram $(B,\Lambda,\vv)$ we construct the symplectic affine variety:
\begin{equation*}
      \TM\coloneqq\TM(B,\Lambda,\vv)\coloneqq \prod\limits_{x \in \Lambda}\M^x \times \bigoplus\limits_{e \in \EE}\M^e.
\end{equation*}
This variety carries a Hamiltonian action of the group $\GG=\prod\limits_{\zeta\in\IS}\GL(v_\zeta)$ given by changing bases (\cite[\S2.2]{NT17}). We denote the moment map by 
\begin{equation*}
    \mu=\mu_{(B,\Lambda,\vv)}: \TM \rightarrow \Lie(\GG)=\bigoplus\limits_{\zeta\in\IS}\mathfrak{gl}(v_\zeta), \ \ \ m \mapsto \bigl(\mu(m)_\zeta\bigr)_{\zeta\in\IS}
\end{equation*}
where $\Lie(\GG)$ and $\Lie(\GG)^*$  are identified via the trace pairing. For each segment $\zeta$,  $\mu(A,B,a,b,C,D)_\zeta$ is given by:
\begin{equation}\label{momentmapequation}
\begin{aligned}
\input{Fig/mu2}.
\end{aligned}
\end{equation}
We want to define bow varieties as Hamiltonian reduction of $\TM$ by the action of $\GG$. However, without loss of generality \cite[\S4]{gaibisso2024quiver}, we will only consider stability and deformation parameters of the following types.
\begin{enumerate}
\item A \textit{deformation parameter} is an element $\lambda \in \CC^\II$ where $\CC^\II$ is embedded in $\CC^{\IS}$ via first segment of intervals and $\CC^{\IS}$ is diagonally embedded in $\Lie(\GG)$.
\item A \textit{stability parameter} is an element $\theta \in \ZZ^\II$ where $\ZZ^\II$ is embedded in the character group of $\GG$, $\ZZ^{\IS}$, via first segment of intervals. Precisely, $\theta=(\theta_\sigma) \in \ZZ^{\II}$ can be regarded as the character
$$\chi_\theta(g)=\prod\limits_{\sigma\in\II}\operatorname{det}(g_{\zeta_\sigma})^{-\theta_\sigma}$$
where $g=(g_\zeta)\in\GG$ and $\zeta_\sigma$ denotes the first segment on the wavy line $\sigma$.
\end{enumerate}
\begin{Definition}
    Given a bow diagram $(B,\Lambda,\vv)$ such that $\vv\in\NN^{\IS}$, we define the associated \textit{bow variety} with deformation parameter $\lambda$ and stability parameter $\theta$ as the Hamiltonian reduction
    \begin{equation*}
    \M=\M_{\lambda,\theta}(B,\Lambda,\vv)=\mu^{-1}(\lambda) \GIT_{\theta} \GG.
\end{equation*}
where $\GIT_{\theta}$ denotes the GIT quotient with respect to the stability parameter $\chi_\theta$.
\end{Definition}
    If $m\in\mu^{-1}(\lambda)$ and $\zeta\coloneqq\zeta_x^+=\zeta_{x'}^-$ for some $x,x'\in\Lambda$, then we define $B_\zeta\coloneqq B_{\zeta_{x}^-}=B_{\zeta_{x}^+}$.
\begin{Definition}
    Let $(B,\Lambda,\vv)$ be a bow diagram. If there exists a segment $\zeta$ such that $v_\zeta<0$, we define $\M=\M_{\lambda,\theta}(B,\Lambda,\vv)=\varnothing$, for every deformation parameter $\lambda$ and every stability parameter $\theta$.
\end{Definition}

\section{Hanany--Witten transitions and supersymmetry inequalities}\label{SectionHW}
\subsection{Hanany--Witten transitions}\label{s.hw}
Given a supersymmetric brane diagram, the Hanany--Witten transition \cite{HW} is a \textit{local} process of brane creation and annihilation that exchanges the positions of two consecutive 5-branes of different types. We can imagine moving a D5-brane through an NS5-brane. If there exists a D3-brane between them, with endpoints on these 5-branes and with winding number $0$, it is annihilated; otherwise, such a brane is created after the transition. All other D3-branes retain the same 5-branes as endpoints, and we can imagine them being lengthened or shortened. We note that Hanany--Witten transitions preserve supersymmetry of brane diagrams.

In terms of bow diagrams such a transition assumes the following form:
\begin{equation}
    \input{Fig/HWdiagrams}\label{HWprocess}
\end{equation}
In view of ~\eqref{HWprocess}, it makes sense to define Hanany--Witten transitions for every local diagram of the type:
\begin{equation}
\input{Fig/figHWcondition}\label{HWcondition}
\end{equation}
Although \cite{NT17} considers only bow diagrams that arise from brane diagrams (not necessarily supersymmetric), this definition makes sense for every bow diagram.
\begin{Remark}\label{Rem:HWandN}
     Define $N_e=v_{h(e)}-v_{t(e)}$ for any arrow $e$ and $N_x=v_x^--v_x^+$ for any x-point $x$. Given an arrow $e$ and an x-point $x$ such that $\zeta_{h(e)}=\zeta_x^-$, by definition (see \cite[Lemma 7.4]{NT17}), the Hanany--Witten transition swapping $e$ and $x$ transforms the numbers $(N_e,N_x)$ to $(N_e-1,N_x-1)$. That is:
    \begin{center}
\input{Fig/hwn}
    \end{center}
\end{Remark}
\begin{Definition}
    Two bow diagrams related by a sequence of Hanany--Witten transitions are said to be \textit{Hanany--Witten equivalent}.
\end{Definition} 
Since the relation between brane diagrams and bow diagrams of affine type A, one may wonder if Hanany--Witten transition has a counterpart in the theory of bow varieties. The answer is given in the following theorem due to Nakajima and Takayama. Note that two bow diagrams related by a sequence of Hanany--Witten transitions share the same underlying bow, and therefore the same space of stability parameters and the same space of deformation parameters (see \S\ref{s.bowvarieties}).
\begin{Theorem}[Nakajima-Takayama, \cite{NT17}, Proposition 7.1]\label{HWtransitionsforbow}
    Let $(B,\Lambda,\vv)$ and $(B,\Lambda',\vv')$ be two bow diagrams related by a sequence of Hanany--Witten transitions. Then,
    \begin{equation*}
        \M_{\lambda,\theta}(B,\Lambda,\vv) \cong \M_{\lambda,\theta}(B,\Lambda',\vv')
    \end{equation*}
for every $\lambda\in \CC^\II$ and $\theta \in \ZZ^\II$, where $\II$ is the set of wavy lines of $B$.
\end{Theorem}
\begin{Remark}\label{Rem:Step1}
     In \cite{NT17}, the previous theorem was stated only for bow diagrams $(B,\Lambda,\vv)$, $(B,\Lambda',\vv')$ with $\vv,\vv'\in\NN^{\IS}$. However, we will see that this result extends to every bow diagram (see \S\ref{s.step1}).
\end{Remark}
In physics, there is no issue with moving a 5-brane through another brane of the same type. However, in this paper, we do not consider such moves. For instance, given a local subdiagram of the form:
\begin{equation*}
\input{Fig/moving2x}
\end{equation*}
if we want to move $\times_2$ through $\bigcirc$, we first apply the Hanany--Witten transition to move $\times_1$ through $\bigcirc$ and then the transition to move $\times_2$ through $\bigcirc$.
\subsection{Supersymmetry inequalities}
\begin{Definition}\label{Definitionseparateddiagram}A \textit{separated} bow diagram of affine type A is a bow diagram $(B,\Lambda,\vv)$ such that all the x-points are on the same wavy line.
\end{Definition}
Given an affine type A separated bow diagram, unless otherwise stated, we label the $x$-points as $1, \dots, w$ in anticlockwise order, starting from the first $x$-point on the wavy line. The arrows are labelled as $1, \dots, n$ in clockwise order, starting from the arrow ending at the origin of the wavy line containing the $x$-points. The entries of the dimension vector are labelled as follows: starting from the first segment of the wavy line with the $x$-points, we assign in anticlockwise order the dimensions $v_0, v_{-1}, \dots, v_{-w}$ to the segments in that wavy line. Then, we denote by $v_s$ the dimension on the last segment of the wavy line ending at the origin of the $s$-th arrow. In particular, $v_{-w}$ and $v_n$ denote the same dimension. See Figures \ref{StypeAffine} and \ref{StypeA}.
\begin{figure}[H]
    \centering
    \input{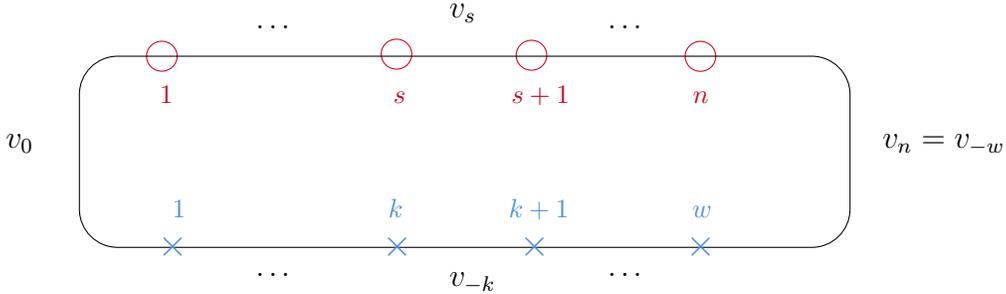}
    \caption{Separated bow diagram of affine type A.}
    \label{StypeAffine}
\end{figure}
\begin{figure}[H]
    \centering
    \input{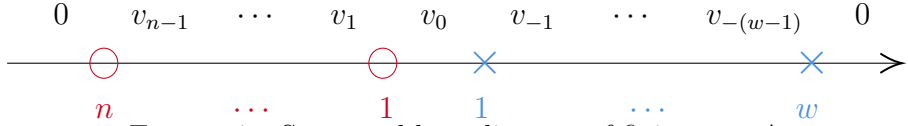}
    \caption{Separated bow diagram of finite type A.}
    \label{StypeA}
\end{figure}
\begin{Remark}
In the cobalanced case (i.e. when for every x-point $x$, $v_x^- = v_x^+$), there is a correspondence between bow diagrams and pairs of quivers and framed dimension vectors, see \cite{cherkis2011instantons}, \cite{NT17}. In this case, the number of $x$-points $w$ corresponds to the sum of the dimensions assigned to framed vertices.
\end{Remark}
\begin{Definition}\label{HWinequalities}
A separated bow diagram of affine type A satisfies the \textit{supersymmetry inequalities} if for every $s\in\{0,1,\dots,n\}$, $k\in\{0,1,\dots,w\}$ and $t\in\NN\setminus\{0\}$
\begin{equation}\label{susyinequalities}
    \begin{aligned}
        \cD_{s,k}^t= &sk+(t-1)(sw+kn)+\frac{(t-1)(t-2)}{2}wn \\
         & +v_s+v_{-k}+(t-1)v_{-w}-tv_0\geq0,\\
        \aD_{n+1-s,w+1-k}^t= & sk+(t-1)(sw+kn)+\frac{(t-1)(t-2)}{2}wn \\
         & +v_{n-s}+v_{-(w-k)}+(t-1)v_0-tv_{-w}\geq0,
        \end{aligned}
    \end{equation}
\end{Definition}
We refer to the former inequalities as the \textit{clockwise supersymmetry inequalities} and to the latter inequalities as the \textit{anticlockwise supersymmetry inequalities}.
\begin{Remark}\label{R:interpretationinequalities}
Let us briefly discuss the meaning of previous inequalities. Details will be provided in \S\ref{s.step2}. 

If $s,k>0$, $\cD^t_{s,k}$ is the dimension between $x_k$ and $\bigcirc_s$ after applying Hanany--Witten transitions to move $x_1,\dots,x_w$ clockwise through $\bigcirc_1,\dots,\bigcirc_{n}$ $(t-1)$-times and further move $x_1,\dots,x_k$ clockwise through $\bigcirc_1,\dots,\bigcirc_s$. Similar considerations hold for the anticlockwise case. That is, $\aD^t_{n+1-s,w+1-k}$ is the dimension between $x_{-(w+1-k)}$ and $\bigcirc_{n+1-s}$ after applying Hanany--Witten transitions to move $x_{-w},\dots,x_1$ anticlockwise through $\bigcirc_n,\dots,\bigcirc_1$ $(t-1)$-times and further move $x_{-w},\dots,x_{-(w+1-k)}$ anticlockwise through $\bigcirc_n,\dots\bigcirc_{n+1-s}$. 

For all $s$ and $k$, we have $\cD^1_{s,0}=\cD^0_{s,w} =v_s$ and $\cD^1_{0,k} = \cD^0_{n,k}= v_{-k}$.

Moreover, we notice that including inequalities where $sk=0$ and $t>1$ does not impose any additional condition on the dimension vector. Indeed, the definitions yield $\cD^t_{0,k} = \cD^{t-1}_{n,k}$ and $\cD^t_{s,0} = \cD^{t-1}_{s,w}$. Similarly, for the anticlockwise inequalities, we have $\aD^1_{n+1,w+1-k}=\aD^0_{1,w+1-k}= v_{-(w-k)}$ and $\aD^1_{n+1-s,w+1}=\aD^0_{n+1-s,1}=v_{n-s}$ for all $s$ and $k$. If $t>1$, then $\aD^t_{n+1,w+1-k} = \aD^{t-1}_{1,w+1-k}$ and $\aD^t_{n+1-s,w+1} = \aD^{t-1}_{n+1-s,1}$.
\end{Remark}

\begin{Remark}\label{Inequalitiesfort=0}
Although they would be well-defined, supersymmetry inequalities~\eqref{susyinequalities} for $t=0$ do not yield any new inequality. Indeed, it follows by a direct computation that
\begin{equation}\label{eq.inequalitiesfort=0} \begin{aligned} \cD^0_{n-s,w-k} &= \aD^1_{n+1-s,w+1-k}, &\quad \aD^{0}_{s+1,k+1} &= \cD^1_{s,k} \end{aligned} \end{equation} for every $s \in \{0,1,\dots,n\}$ and $k \in \{0,1,\dots,w\}$.

For a later purpose (see Appendix \ref{InverseStep2-Appendic}), let us provide an interpretation of the identities $\cD^0_{n-s,w-k} = \aD^1_{n+1-s,w+1-k}$. Consider a separated bow diagram $(B,\Lambda,\vv)$ as in Figure~\ref{StypeAffine}, and apply Hanany--Witten transitions to move all the $x$-points clockwise through all the arrows exactly $t>0$ times. The resulting bow diagram is given by:

\begin{equation*} \input{Fig/propstep2} \end{equation*}

Next, move $x_w$ anticlockwise through the segments $\bigcirc_n, \dots, \bigcirc_s$. The final diagram can equivalently be obtained by:
\begin{itemize}
   \item first moving all the $x$-points clockwise through all the arrows $(t-1)$ times,

\item then moving $x_1, \dots, x_{w-1}$ once more clockwise through all the arrows,

    \item and finally moving $x_w$ clockwise through $\bigcirc_1, \dots, \bigcirc_{s-1}$.
\end{itemize}
This leads to the diagram: \begin{equation}\label{Fig:inequalitiesfort=0} \input{Fig/inequalitiest=0} \end{equation}

The previous sequence of Hanany--Witten transitions makes sense for $t=0$. That is, suppose we only move $x_w$ anticlockwise through $\bigcirc_n,\dots,\bigcirc_s$. Then, the equalities in~\eqref{eq.inequalitiesfort=0} implies that the resulting bow diagram is given by \eqref{Fig:inequalitiesfort=0} with $t=0$. 

The remaining inequalities follow from analogous arguments, with anticlockwise transitions replacing clockwise ones.
\end{Remark}

%% file: Fig/GenTri.tex
\begin{tikzcd}
\CC^{v_x^-} \arrow["B_x^-", loop, distance=2em, in=55, out=125] \arrow[rr, "A_x"] \arrow[rd, "b_x"'] &                             & \CC^{v_x^+} \arrow["B_x^+", loop, distance=2em, in=55, out=125] \\
                                                                                                    & \mathbb{C} \arrow[ru, "a_x"'] &                                                                   
\end{tikzcd}

%% file: 002CriterionForStratum.tex
\section{Weights of affine Lie algebras and the stratum condition}\label{SectionWandStratum}
In this section, unless otherwise stated, we assume that all bow diagrams have at least one arrow and one x-point.
\subsection{Balanced bow diagrams and weights}\label{s.balance}
We review the notion of balanced bow diagrams and their connection with weights of affine Lie algebras. Appendix \ref{s.weights} provides the necessary notation and background on affine Lie algebra weights and generalized Young diagrams. A detailed study of balanced bow varieties can be found in \cite[Section 7]{NT17}.
\begin{Definition}
    A bow diagram is said to be balanced if and only if for every arrow $e$, we have $v_{t(e)}=v_{h(e)}$.
\end{Definition}
For a later purpose, we consider x-points indexed by $i\in\ZZ/\ZZ w$, so that $x_w=x_0$.

 Let us consider a balanced bow diagram:
\begin{equation}\label{balancediagram}
    \input{Fig/balancedaffinediagram}
\end{equation}
 To such a diagram one can associate a pair of $(\Liesl_w)_{\text{aff}}$-weights (\cite[\S7.6]{NT17}) as follows:
\begin{equation*}
    \lambda=\sum\limits_{i=0}^{w-1}d_i\Lambda_i+v_0\delta, \ \ \ \mu=\sum\limits_{i=0}^{w-1}(d_i\Lambda_i-v_i\alpha_i)+v_0\delta.
\end{equation*}
This establishes a correspondence between balanced bow diagrams with $w$ x-points and $n$ arrows and pairs $(\lambda,\mu)$ of $(\Liesl_w)_{\text{aff}}$-weights of level $n$ such that: $\langle\mu,d\rangle=0$, $\lambda$ is dominant and $\lambda-\mu\in\bigoplus\limits_{i=0}^{w-1}\ZZ\alpha_i$. We will denote the balanced bow diagram in ~\eqref{balancediagram} by $\mathcal{B}(\lambda,\mu)$.

A separated bow diagram (Definition \ref{Definitionseparateddiagram}) is uniquely determined by a pair of integral vectors $\bigl([\prescript{t}{}{\lambda}_s],[\mu_i]\bigr)\in\ZZ^n\times\ZZ^w$ and a number $v\in\ZZ$ as follows. In the notation of Figure \ref{StypeAffine}, we set $\prescript{t}{}{\lambda}_{s}=N_{e_s}=v_{s-1}-v_{s}$ for every $s\in\{1,\dots,n\}$, $\mu_{i}=N_{x_i}=v_{-(i-1)}-v_{-i}$ for every $i\in\ZZ/w\ZZ$ and $v=v_n=v_{-w}$. We denote such a bow diagram by $\mathcal{B}(\prescript{t}{}{\lambda}_s,\mu_i,v)$ and will represent it as follows:

\begin{equation}\label{Fig:separatedweights}
    \input{Fig/separatedweights}
\end{equation}

 By construction, $[\prescript{t}{}{\lambda}_s]$ and $[\mu_i]$ have the same charge (see \S\ref{s.A2}): 
$$v+\sum\limits_{s=1}^{n}\prescript{t}{}{\lambda}_s=v+\sum\limits_{i=1}^{w}\mu_i.$$
\begin{Remark}\label{remarksupersymmetryforbalance}
   Every balanced bow diagram $(B,\Lambda,\vv)$ with $\vv\in\NN^{\IS}$ originates from a brane diagram with all the D3-branes stretched between D5-branes (i.e. x-points). It follows by the definition of supersymmetric bow diagrams that a balanced diagram $(B,\Lambda,\vv)$ is supersymmetric if and only if $\vv\in\NN^{\IS}$. Hence, $\mathcal{B}(\lambda,\mu)$ is supersymmetric if and only if $\lambda\geq\mu$ in the dominance order.
\end{Remark} 
\begin{Proposition}(Proposition 7.19 \cite{NT17})\label{Uniquebalancetheorem}
    There is at most one balanced bow diagram in each Hanany--Witten equivalent class of bow diagrams.
\end{Proposition}
\begin{Proposition}[\cite{NT17}]\label{separatedtobalancetheorem}
  A separated bow diagram $\mathcal{B}(\prescript{t}{}{\lambda}_s,\mu_i,v)$ is Hanany--Witten equivalent to a balanced bow diagram if and only if $[\prescript{t}{}{\lambda}_s]\in\Y_n^w$ (see \S\ref{s.A3}). 
  In this case, the $(\Liesl_w)_{\text{aff}}$-weights associated with such a balanced bow diagram are given by:
\begin{equation}\label{equationintegraltoweights}
  \begin{gathered}
      \lambda=(n-\lambda_1+\lambda_w)\Lambda_0+\sum\limits_{i=1}^{w-1}(\lambda_i-\lambda_{i+1})\Lambda_i+\hat{v}\delta, \\
      \mu=(n-\mu_1+\mu_w)\Lambda_0+\sum\limits_{i=1}^{w-1}(\mu_i-\mu_{i+1})\Lambda_i
  \end{gathered}
\end{equation}
where $[\lambda_i]\in\Y_w^n$ is the transposed of $[\prescript{t}{}{\lambda}_s]$ and $\hat{v}$ is uniquely determined by the original diagram. More precisely, $\hat{v}$ is the dimension on the segment $\zeta_{x_0}^+$ in the Hanany--Witten equivalent balanced bow diagram.
\end{Proposition}
\begin{proof}
    From \cite[Proposition 7.5, Lemma 7.18, Proposition 7.20]{NT17}, we know that $\mathcal{B}(\prescript{t}{}{\lambda}_s,\mu_i,v)$ is Hanany--Witten equivalent to a balanced bow diagram if and only if $[\prescript{t}{}{\lambda}_s]\in\Y_n^w$.  The second statement follows from \cite[Lemma 7.18 and Proposition 7.19]{NT17} and general properties of generalized Young diagrams. Let us notice that $(\lambda,\mu)$ corresponds to a balanced bow diagram. Indeed, $\lambda$ is dominant, $\lambda$ and $\mu$ have the same level, $\langle\mu,d\rangle=0$ and, by ~\eqref{eq.changingweightbasis}, $\lambda-\mu\in\bigoplus\limits_{i=0}^{w-1}\ZZ\alpha_i$.
\end{proof}
\begin{Corollary}\label{Corollaryseparatedbalanced} A pair $(\lambda, \mu)$ of $(\Liegl_w)_{\text{aff}}$-weights satisfying the following conditions:
\begin{enumerate}
   \item $\lambda$ is dominant,
    \item $\lambda$ and $\mu$ have the same charge and the same level,
   \item $\langle d, \mu \rangle = 0$,
\end{enumerate}
is equivalent to the data of a separated bow diagram $\mathcal{B}(\prescript{t}{}{\lambda}_s,\mu_i,v)$ Hanany-–Witten equivalent to the balanced bow diagram corresponding to the $(\Liesl_w)_{\text{aff}}$-weights obtained by projecting $(\lambda,\mu)$.
\end{Corollary}
\begin{proof}
 Given $(\lambda,\mu)$ as in the assumptions, for every $v\in\ZZ$ we have a well-defined bow diagram $\mathcal{B}(\prescript{t}{}{\lambda}_s,\mu_i,v)$  Hanany--Witten equivalent to a balanced bow diagram (Proposition \ref{separatedtobalancetheorem}). Using ~\eqref{equationintegraltoweights}, thanks to $(1), (2)$ and $(3)$, we have a well-defined balanced bow diagram $\mathcal{B}(\lambda,\mu)$, where, by abuse of notation, here $\lambda,\mu$ have to be regarded as $(\Liesl_w)_{\text{aff}}$-weights (Remark \ref{projectionmodulodelta}). Finally, $v$ is uniquely determined by requiring $\mathcal{B}(\prescript{t}{}{\lambda}_s,\mu_i,v)$ Hanany--Witten equivalent to $\mathcal{B}(\lambda,\mu)$. The other direction is an immediate consequence of Proposition \ref{separatedtobalancetheorem}. 
\end{proof}

\begin{Example}\label{exseparatedtobalance}
Consider a separated bow diagram $\mathcal{B}(\prescript{t}{}{\lambda}_s,\mu_i,v)$ such that
\[
[\prescript{t}{}{\lambda}_s] \in \Y_n^w \quad \text{and} \quad w > \prescript{t}{}{\lambda}_1 \geq 0.
\]
Then, $\hat{v}$ in Proposition \ref{separatedtobalancetheorem} is given by $v+\sum\limits_{s=n-m+1}^{n}\prescript{t}{}{\lambda_s}$ where $m=\vert\{s=1,\dots,n \ \vert \ \prescript{t}{}{\lambda}_s<0\}\vert$. Indeed, if we move $e_{n+1-m},\dots,e_n$ clockwise to the left of $e_1$, thanks to Remark \ref{Rem:HWandN} (we have oriented bow diagrams anticlockwise), we obtain:
    \begin{center}
        \input{Fig/Exampleweightsdiagram}
    \end{center}
Note that $-\sum\limits_{s=n-m+1}^{n}\prescript{t}{}{\lambda_s}$ is equal to the number of white boxes in the block $N=-1/2$ of the Maya diagram associated with $[\prescript{t}{}{\lambda}_s]$.
Let us denote by $[\prescript{t}{}{\hat{\lambda}_s]}$ the integral vector associated with such a configuration; that is:
\begin{equation*}
	\prescript{t}{}{\hat{\lambda}_s}=\begin{cases}
		\lambda_{n-m+s}+w & \text{if }s=1,\dots,m,\\
		\lambda_{s-m} & \text{if }s=m+1,\dots,n.
	\end{cases}
\end{equation*}
Therefore:
\begin{equation*}
    w>\prescript{t}{}{\hat{\lambda}_1}\geq\cdots\geq\prescript{t}{}{\hat{\lambda}_n}\geq 0.
\end{equation*}
Since $\prescript{t}{}{\hat{\lambda}}_1\geq\cdots\geq\prescript{t}{}{\hat{\lambda}}_n\geq0\geq\prescript{t}{}{\hat{\lambda}}_1-w$, the bow diagram is transformed into a balanced bow diagram by successive application of Hanany--Witten transitions. In particular, since $\prescript{t}{}{\hat{\lambda}}_s\geq0$ for every $s$, we need to move the s-th arrows in the new configuration anticlockwise through $\prescript{t}{}{\hat{\lambda}}_{s}$ x-points. Since $w > \prescript{t}{}{\hat{\lambda}}_1$, we never cross $x_0$ and, therefore, the claim.
\end{Example}

Let us consider the finite case. As expected, in this case we only need weights of Lie algebras of finite type. 

Given a balanced bow diagram of finite type A
\begin{equation}\label{finitebalanced}
\input{Fig/finitebalanced}
\end{equation}
we define two $\Liegl_w$-weights by:
\begin{align}\label{weightsfinitetype}
    \lambda=\sum\limits_{i=1}^{w}d_i\Lambda_i, & & \mu=\sum\limits_{i=1}^{w-1}(d_i\Lambda_i-v_i\alpha_i)+d_w\Lambda_w. 
\end{align}
This gives a one-to-one correspondence between balanced bow diagrams of finite type A with $n$ arrows and $w$ x-points and pairs of $\Liegl_w$-weights $(\lambda,\mu)$, such that $\lambda$ is polynomial dominant, the level of $\lambda$ is at most $n$ and they have the same charge. 

With respect to the canonical basis for the Cartan subalgebra of $\Liegl_w$ given by the diagonal matrices, the weights in ~\eqref{weightsfinitetype} correspond to the integral vectors $[\lambda_i],[\mu_i]\in\ZZ^w$ defined as follows:
\begin{align*}
  &\lambda_i=\sum\limits_{j=i}^{w}d_j & i=1,\dots,w & &\\  &\mu_i=v_{w-1}+d_w+\sum\limits_{j=i}^{w-1}u_j & i=1,\dots,w-1, & &
  \mu_w=v_{w-1}+d_w
\end{align*}
with $u_i=d_i-2v_i+v_{i-1}+v_{i+1}$. See \S\ref{s.A1} for more details.

This provides us with a correspondence between balanced finite type A bow diagrams with $n$ arrows and $w$ x-points and pairs of integral vectors $([\lambda_i],[\mu_i])\in\ZZ^w\times\ZZ^w$ such that $[\lambda_i]$ is a usual Young diagram (with $w$ rows), $n\geq\lambda_1$ and $\sum\limits_{i=1}^{w}\lambda_i=\sum\limits_{i=1}^{w}\mu_i$. By considering $[\lambda_i]$ as a Young diagram of $\sum\limits_{i=1}^{w}\lambda_i$ boxes into a block of size $w\times n$ and transposing it, the previous data is equivalent to a pair of integral vectors $([\prescript{t}{}{\lambda}_s],[\mu_i])\in\ZZ^n\times\ZZ^w$ such that $[\prescript{t}{}{\lambda}_s]$ is a (usual) Young diagram with $n$ rows, $w\geq\prescript{t}{}{\lambda}_1$ and $\sum\limits_{s=1}^{n}\prescript{t}{}{\lambda}_s=\sum\limits_{i=1}^{w}\mu_i$.

Finally, such a pair of data is equivalent to giving the separated bow diagram:
\begin{equation}\label{finitebalancedseparated}
\input{Fig/finitebalancedseparated}
\end{equation}
It follows by construction that the bow diagram ~\eqref{finitebalancedseparated} can be obtained from the bow diagram ~\eqref{finitebalanced} by applying Hanany--Witten transitions to move all the arrows to the left of $x_1$.

\begin{Remark}
There is a slight difference between our construction and the approach taken in \cite{NT17} to provide a correspondence between finite type A balanced bow diagrams and weights of a finite type Lie algebra. However, if we project our $\mathfrak{gl}_w$-weights $(\lambda, \mu)$ onto the lattice of $\mathfrak{sl}_w$-weights, we recover the weights used in \cite{NT17}. Specifically, they consider balanced bow diagrams ~\eqref{finitebalanced} with $n = d_1 + \cdots + d_{w-1}$. This assumption is not restrictive, as the additional arrows we introduce do not affect bow varieties. However, keeping track of these extra arrows will be useful for our purposes.
\end{Remark}

\subsection{The stratum condition}\label{s.stratumcriterion}
One of the primary objectives of this paper is to characterize bow diagrams that generate non-empty bow varieties. It is well-known \cite[Section 4]{NT17} that bow varieties are naturally stratified by the conjugacy classes of stabilizers in the sense of Sjamaar—Lerman \cite{SL91}. For balanced bow diagrams, Nakajima and Takayama \cite[Theorem 7.26]{NT17} parametrized the strata of $\M_{0,0}$ in terms of the weights of affine Lie algebras, a result that can be generalized to all bow diagrams of affine type A \cite[Remark 7.27]{NT17}. In this subsection, we introduce a condition for bow diagrams, which can be regarded as an equivalent condition for the existence of a stratum in $M_{0,0}$ (Remark \ref{Rem:stratumconditionexistencestratum}). While we will not provide a detailed proof of such an equivalence, we will demonstrate in \S\ref{s.step5} that this condition is equivalent to supersymmetry inequalities (Definition \ref{Def-susyincrements}). Additionally, we will discuss the ``S-dual” of this condition in Remark \ref{remarkSdualstratumcondition}.

\begin{Definition}\label{def:stratumcondition} We say that a bow diagram satisfies the \textit{stratum condition} if any Hanany--Witten equivalent separated bow diagram $\mathcal{B}(\prescript{t}{}{\lambda}_s,\mu_i,v)$ admits a supersymmetric subdiagram $\mathcal{B}(\prescript{t}{}{\kappa}_s,\mu_i,v')$ which is Hanany--Witten equivalent to a balanced bow diagram.
\end{Definition}
\begin{Remark}
For bow diagrams without arrows or without x-points, the definition of the stratum condition is still well-defined and it is always trivially satisfied.
\end{Remark}

By definition, Hanany--Witten transitions preserve the stratum condition. It is important to note that the stratum condition is not equivalent to requiring that $\mathcal{B}(\prescript{t}{}{\lambda}_s,\mu_i,v)$ admits a supersymmetric subdiagram that is Hanany–Witten equivalent to a balanced one. The key point is that we cannot modify $[\mu_i]$.

\begin{Remark}\label{Rem:stratumconditionexistencestratum}
Let us briefly explain the equivalence between the stratum condition and the existence of a stratum in $ \mathcal{M}_{0,0}$. 

Consider a separated bow diagram $\mathcal{B}(\prescript{t}{}{\lambda}_s, \mu_j, v)$ and let us denote by $\mu$ the $(\Liegl_w)_{\text{aff}}$-weight given by $[\mu_i]$ such that $\langle\mu,d\rangle=0$. Nakajima and Takayama's argument \cite[Theorem 7.26, Remark 7.27]{NT17} provides us with a description of strata in terms of pairs $(\underline{k},\kappa)$ where $ \underline{k}$ is a partition and $\kappa$ is a $(\Liegl_w)_{\text{aff}}$-weight such that:
\begin{enumerate}
    \item $\kappa$ is dominant
    \item $\kappa$ and $\mu$ have the same charge and the same level
    \item $\kappa\geq\mu$ in the dominance order
    \item The corresponding separated bow diagram $\mathcal{B}(\prescript{t}{}{\kappa}_s,\mu_i,v')$ (see Corollary \ref{Corollaryseparatedbalanced}) is a subdiagram of $\mathcal{B}(\prescript{t}{}{\lambda}_s, \mu_j, v-|\underline{k}|)$.
\end{enumerate}

By definition, the stratum condition is equivalent to require a pair as above such that $|\underline{k}|=0$. However, in general, not all the pairs $(\underline{k},\kappa)$ as above corresponds to a stratum. The existence of a stratum for the pair $(\underline{k},\kappa)$ is equivalent to the non-emptiness of the stable locus of the bow variety $\M_{0,0}$ associated with the balanced bow diagram $\mathcal{B}(\kappa,\mu)$, where here we are considering $\kappa,\mu$ as $(\Liesl_w)_{\text{aff}}$-weights. See \cite[Theorem 7.26]{NT17}.

We claim that the stratum condition is equivalent to the existence of a stratum.

If $n,w>1$, every pair $(\underline{k},\kappa)$ as defined above corresponds to a stratum and so the claim is proved.

Suppose $ n > 1 $ and $ w = 1 $. In this case, we have a one-to-one correspondence between strata and pairs $(\underline{k},\kappa)$ as above satisfying the further assumption $\kappa=\mu$. This is equivalent to require that the bow subdiagram $\mathcal{B}\left( \prescript{t}{}{\kappa}_s, \mu_j, v' \right)$ is Hanany--Witten equivalent to a (balanced) bow diagram with dimension vector $0$. However, since $w=1$, the existence of such a subdiagram is equivalent to the stratum condition.

If $n = 1$, then the bow diagram is Hanany--Witten equivalent to a balanced one (see Remark \ref{Rem:HWandN}). In this case, both the stratum condition and the existence of a stratum clearly reduce to supersymmetry. 

In the case $w=1$, the bow diagram is Hanany--Witten equivalent to a cobalanced one and the existence of a stratum is easily reducible to supersymmetry (see Remark \ref{Remarksusyforco-balanced}). Indeed, the case $n=1$ and the case $w=1$ are related by S-duality which preserves supersymmetry (see Remark \ref{remarkSdualstratumcondition} for details).
\end{Remark}

As for supersymmetry and supersymmetry inequalities, we will translate the stratum condition for separated bow diagrams into a numerical criterion.
Let us notice that every bow diagram is Hanany--Witten equivalent to a separated bow diagram $\mathcal{B}(\prescript{t}{}{\lambda},\mu,v')$ such that $w >\vert\prescript{t}{}{\lambda}\vert\geq 0$. This will be proved later, see Lemma \ref{Lemma.reducing}.
\begin{Lemma}\label{Affinestratumcondition}
     Consider a separated bow diagram $\mathcal{B}(\prescript{t}{}{\lambda}_s,\mu_i,v)$ satisfying the condition $w >\vert\prescript{t}{}{\lambda}\vert\geq 0$. Then, it satisfies the stratum condition if and only if there exists a dominant $(\Liegl_w)_{\text{aff}}$-weight $\kappa$ of level $n$ such that 
    \begin{enumerate}
        \item $\vert \kappa \vert=\vert \mu\vert$,
        \item $\kappa \geq \mu$,
        \item $\prescript{t}{}{\kappa}+\bigl(\sum\limits_{s=n+1-m}^{n}\prescript{t}{}{\kappa}_s\bigr)\delta \geq \prescript{t}{}{\lambda}$, where $m=\vert\{s=1,\dots,n \ \vert\ \prescript{t}{}{\kappa}_s<0\}\vert$ and $\prescript{t}{}{\lambda}$ is the weight given by $[\prescript{t}{}{\lambda}_s]$ with $\langle\prescript{t}{}{\lambda},d\rangle=-v$.
    \end{enumerate}
\end{Lemma}
\begin{proof}
    From \S\ref{s.balance}, we know that $\kappa,\mu$ as above satisfying conditions $(1)$ and $(2)$ is equivalent to have a separated bow diagram $\mathcal{B}(\prescript{t}{}{\kappa}_s,\mu_i,v')$
    \begin{equation*}
        \input{Fig/separateddiagramstratum}
    \end{equation*}
which is supersymmetric and Hanany--Witten equivalent to a balanced bow diagram. We conclude if we show that, under the previous assumption, the new bow diagram satisfies condition $(3)$ if and only if it is a subdiagram of $\mathcal{B}(\prescript{t}{}{\lambda}_s,\mu_i,v)$. We define $\hat{v}=\langle\kappa,d\rangle$. Since we are dealing with weights having the same charge, condition $(3)$ means:
\begin{equation*}
 \begin{aligned}   \Bigl(\hat{v}-\sum\limits_{s=n+1-m}^{n}\prescript{t}{}{\kappa}_s\Bigr)+\sum\limits_{i=j}^{n}\prescript{t}{}{\kappa_i}\leq v+\sum\limits_{i=j}^{n}\prescript{t}{}{\lambda}_i & & j\in\{1,\dots,n\}.
 \end{aligned}
\end{equation*}
The claim follows by proving that $\hat{v}-\sum\limits_{s=n=1-m}^{n}\prescript{t}{}{\kappa}_s=v'$. Indeed, $\vert\prescript{t}{}{\lambda}\vert\geq0$ implies that $\prescript{t}{}{\kappa_1}\geq0$. In addition, we have:
\begin{equation*}
    w>\vert\prescript{t}{}{\lambda}\vert=\vert\prescript{t}{}{\kappa}\vert\geq \prescript{t}{}{\kappa}_1+(n-1)\prescript{t}{}{\kappa}_n\geq n\prescript{t}{}{\kappa}_1-(n-1)w.
\end{equation*}
Therefore, $w>\prescript{t}{}{\kappa}_1\geq0$. The claim follows from Example \ref{exseparatedtobalance}.
\end{proof}

For later purpose, let us provide a numerical characterization of the stratum condition for separated bow diagrams of finite type.

\begin{Lemma}
Let us consider a separated bow diagram of finite type A as in ~\eqref{finitebalancedseparated}. Let $\prescript{t}{}{\lambda}$ and $\mu$ be, respectively, the $\Liegl_w$ and $\Liegl_n$-weights corresponding to $[\prescript{t}{}{\lambda}_s]$ and $[\mu_j]$. Then, the diagram satisfies the stratum condition if and only if there exists a polynomial dominant $\Liegl_w$-weight, $\kappa$ such that 
\begin{enumerate}
    \item $\vert \kappa\vert=\vert\mu\vert$ and $n\geq\kappa_1$
    \item $\kappa\geq\mu$
    \item $\prescript{t}{}{\kappa}\geq\prescript{t}{}{\lambda}$
\end{enumerate}
\end{Lemma}
\begin{proof}
   Let us notice that $(1)$ means $\kappa$ and $\mu$ have the same charge and the level of $\kappa$ is less than or equal to $n$. It follows from \S\ref{s.balance} that the existence of a polynomial dominant weight $\kappa$ satisfying condition $(1)$ and $(2)$ is equivalent to the existence of a separated diagram:
    \begin{equation*}
        \input{Fig/finitebalancedseparatedsubdiagram}
    \end{equation*}
which is supersymmetric and Hanany--Witten equivalent to a balanced bow diagram. Since $\vert\prescript{t}{}{\lambda}\vert=\vert\kappa\vert=\vert\prescript{t}{}{\kappa}\vert$, it follows by a direct computation that condition $(3)$ is equivalent to require that it is a subdiagram of ~\eqref{finitebalancedseparated}.
\end{proof}
\begin{Remark}
The stratum condition for bow diagrams Hanany--Witten equivalent to a balanced bow diagram can be further simplified. Suppose to be in the finite case and $[\prescript{t}{}{\lambda}_s]$ is a Young diagram, we can rewrite conditions $(1)$, $(2)$ and $(3)$ into $\lambda\geq\kappa\geq\mu$. Indeed, by the definition of dominance order on Young diagrams (with the same number of blocks), we have that $\prescript{t}{}{\kappa}\geq\prescript{t}{}{\lambda}$ is equivalent to $\kappa\leq\lambda$. These inequalities are exactly those that appear in the stratification of finite type A balanced bow varieties \cite[Theorem 7.13]{NT17}. Similarly, suppose to have an affine type A bow diagram which is Hanany--Witten equivalent to a balanced bow diagram $\mathcal{B}(\lambda,\mu)$. Then the stratum condition is equivalent to the existence of a dominant weight $\kappa$ such that $\lambda\geq\kappa\geq\mu$ as in \cite[Theorem 7.26]{NT17}.
\end{Remark}

%% file: 003Supersymmetryandbowvarieties.tex
\section{The proof of the main theorem}\label{section3}
In this section, we prove the main theorem of the paper. We begin by stating the theorem and providing an outline of the proof.
\begin{Theorem}\label{maintheorem}
    Given a bow diagram $(B,\Lambda,\vv)$ of affine type A, the following are equivalent.
\begin{enumerate}[label=(\alph*)]
\item $\M_{0,0} \neq \varnothing$.
    \item There exists $(\lambda,\theta) \in \CC^\II\times\ZZ^\II$ such that $\M_{\lambda,\theta}\neq\varnothing$.
    \item Every Hanany--Witten equivalent bow diagram originates from a brane diagram.
    \item Any Hanany--Witten equivalent separated bow diagram satisfies supersymmetry inequalities.
    \item It is supersymmetric.
    \item It satisfies the stratum condition.
\end{enumerate}
\end{Theorem}
\begin{Remark}
   According to Definition \ref{definitionbowdiagram}, condition $(c)$ can be rephrased as: if a bow diagram $(B,\Lambda',\vv')$ is obtained from $(B,\Lambda,\vv)$ by successive applications of Hanany--Witten transitions, then $v'_\zeta\geq0$ for every segment $\zeta$. In other terms, we cannot generate negative dimensions by applying a sequence of Hanany--Witten transitions.  We notice that $(c) \Leftrightarrow (d) \Leftrightarrow (e)$ provides a numerical criterion for supersymmetry of affine type A brane diagrams.
\end{Remark}

Outline of the proof:
\begin{itemize}
    \item The implication $(a) \Rightarrow (b)$ is trivial.
    
    \item The implication $(e) \Rightarrow (c)$ follows from the definition of Hanany--Witten transitions for supersymmetric brane systems as a process of brane creation and annihilation. See \S\ref{s.hw} or the original work \cite{HW}.

    The implication $(e) \Rightarrow (c)$ follows from the original definition of Hanany--Witten transitions for supersymmetric brane systems as a process of brane creation and annihilation (see \S\ref{s.hw} or the original work \cite{HW}).
    
    \item The proof of the remaining equivalence is divided into five main steps, which are summarized in the following diagram.
\end{itemize}
\begin{center}
\input{Fig/Summaryproof}
\end{center}

The equivalence $(a) \Leftrightarrow (f)$ can be deduced from the interpretation of the stratum condition as the existence of a stratum in $\M_{0,0}$ (see Remark~\ref{Rem:stratumconditionexistencestratum}). However, familiarity with the stratification of bow varieties is not required, as we will instead provide a proof of the equivalence $(d) \Leftrightarrow (f)$. 

Also, in Appendix \ref{InverseStep2-Appendic} we provide a direct proof of $(c)\Leftarrow(d)$. Although this is not essential for the formal proof of the main theorem, it is interesting in its own right and, as far as we know, introduces new ideas for carrying out computations with Hanany--Witten transitions in affine type A bow diagrams.

\begin{Remark}
  Let us consider a bow diagram $(B,\Lambda,\vv)$ with no arrows or no x-points (i.e., $nw=0$). 
    Thanks to Theorem \ref{supersymmetricincrementsforbows}, the moment map equation $\mu=0$ has a solution if and only if $v_\zeta\geq0$ for every segment $\zeta$. Also, by definition of bow varieties, if $(b)$ holds then $v_\zeta\geq0$ for every segment $\zeta$. Therefore $(a)$ and $(b)$ hold if and only if $(B,\Lambda,\vv)$ originates from a brane diagram. It follows by construction that conditions $(c),(d),(e),(f)$ are satisfies if and only if  $v_\zeta\geq0$ for every segment $\zeta$.
\end{Remark}

\begin{Remark}\label{remarkSdualstratumcondition}
Supersymmetry is manifestly preserved under S-duality for affine type A brane systems, which consists in replacing every NS5-brane with a D5-brane and vice versa (see \cite[\S II.D]{GK99}, \cite[\S 7.1]{NT17}). In terms of bow diagrams, the S-duality replaces every arrow with an x-point and vice versa. As a consequence, all the equivalent conditions in Theorem~\ref{maintheorem} has to be preserved under S-duality.

The S-dual of the balanced condition is the \emph{cobalanced condition}, which requires that $v_{\zeta_x^-} = v_{\zeta_x^+}$ for every x-point $x$. Therefore, a bow diagram satisfies the \emph{S-dual stratum condition} if any Hanany--Witten equivalent separated bow diagram of the form $\mathcal{B}(\prescript{t}{}{\lambda}_s,\mu_i,v)$ admits a supersymmetric subdiagram $\mathcal{B}(\prescript{t}{}{\lambda}_s,\kappa_i,v')$ which is Hanany--Witten equivalent to a cobalanced bow diagram.

While in the original stratum condition we look for a subdiagram determined by a generalized Young diagram $[\prescript{t}{}{\kappa}_s]$, in the S-dual stratum condition we are instead looking for the transpose of such a generalized Young diagram. That is, in this case, S-duality corresponds to the transposition of generalized Young diagrams, which is related to the level-rank duality for Kac–-Moody Lie algebras of affine type A (see \S \ref{s.A3}).

We expect this duality to reflect a deeper geometric correspondence between strata and slices. As observed in Remark~\ref{Rem:stratumconditionexistencestratum}, a subdiagram $\mathcal{B}(\prescript{t}{}{\kappa}_s,\mu_i,v')$ as in the stratum condition determines a stratum of the bow variety $M_{0,0}$. Motivated by the stratification of balanced bow varieties in \cite{NT17}, we expect that a subdiagram $\mathcal{B}(\prescript{t}{}{\lambda}_s,\kappa_i,v')$ as in the S-dual stratum condition determines a slice to the stratum corresponding to $\mathcal{B}(\prescript{t}{}{\kappa}_s,\mu_i,v')$.

\end{Remark}

\begin{Remark}\label{Remarksusyforco-balanced}
It is known that bow varieties associated with supersymmetric cobalanced bow diagrams are Nakajima quiver varieties \cite{cherkis2011instantons}, while those associated with supersymmetric balanced bow diagrams are Coulomb branches of supersymmetric quiver gauge theories of affine type A \cite[\S6]{NT17}. As a consequence of Theorem \ref{maintheorem}, we recover the well-known fact that non-deformed Nakajima quiver varieties and Coulomb branches of affine type A supersymmetric quiver gauge theories are non-empty.
\end{Remark}

\begin{Remark}
Although supersymmetry inequalities are natural from the viewpoint of seeing what is required for the Hanany--Witten transitions to not lose the correspondence between bow and brane diagrams. However, there is an alternative approach for proving $(b)\Rightarrow(e)$ in the finite type A case. This approach involves fixing a solution of the moment map equation and studying it to construct a supersymmetric brane diagram. However, we do not present such a method in this paper since we have been able to apply it only in the finite type A case. Extending this method to more general settings remains an intriguing question for future work.
\end{Remark}

\subsection{Proof of Step 1}\label{s.step1}
We prove that if a bow diagram generates a non-empty bow variety, then every Hanany--Witten equivalent bow diagram originates from a brane diagram.

By definition of bow varieties, if a bow diagram generates a non-empty bow variety, it originates from a brane diagram.

Let us recall a necessary condition for affine type A bow diagrams to generate a non-empty bow variety:

\begin{Lemma}\label{LEMMASW} If a bow diagram admits a non-empty bow variety, then the application of any Hanany--Witten transition does not produce a negative dimension. \end{Lemma}

We omit the proof here. For $\lambda=0$, this was established in \cite[Corollary 3.14]{SW23}, and for general $\lambda$, it can be derived from Nakajima and Takayama’s proof of Theorem \ref{HWtransitionsforbow} (\cite[Proposition 7.1]{NT17}). For a proof, we refer the reader to \cite[Corollary 4.17]{gaibisso2024quiver}, where \cite[Corollary 3.14]{SW23} is generalized to bow varieties of general type and for arbitrary deformation parameter $\lambda$.

Step 1 follows from the previous lemma by observing that Hanany–-Witten transitions induce isomorphisms of bow varieties (Theorem \ref{HWtransitionsforbow}) and thus preserve the property of bow diagrams of generating a non-empty bow variety.

\begin{Remark}
As shown in \cite{gaibisso2024quiver}, the condition in Lemma \ref{LEMMASW} is not a sufficient condition. That is, it is not enough to check sequence of Hanany--Witten transitions consisting of one transition. For instance, consider the bow diagram 
\begin{equation*}
    \input{Fig/notsusyexample}
\end{equation*}
Although we can apply the Hanany--Witten transition to swap $\bigcirc_i$ and $x_i$ for $i=1,2$, the bow diagram is not supersymmetric.
\end{Remark}

\subsection{Proof of Step 2}\label{s.step2}
Consider a separated bow diagram $(B,\Lambda,\vv)$ as in Figure \ref{StypeAffine}. We prove that if every Hanany--Witten equivalent bow diagram originates from a brane diagram, then it satisfies the supersymmetry inequalities (\ref{susyinequalities}). 

Suppose $n = 0$ or $w = 0$, the supersymmetry inequalities reduce to the condition $v_\zeta \geq 0$ for all segments $\zeta$. This is equivalent to the condition for $(B, \Lambda, \vv)$ to originate from a brane diagram. Moreover, in this case, the only Hanany--Witten equivalent bow diagram to $(B, \Lambda, \vv)$ is $(B, \Lambda, \vv)$ itself. Therefore, if $nw=0$, it is clear that $(c)\Leftrightarrow(d)$.

From now on, let us assume $n,w>0$.

Let $t\in\NN^+$, $s \in \{0,1,\dots,n\}$ and $k\in\{0,1,\dots,w\}$. If $s,k\neq0$, we define $\cd^t_{s,k}$ as the dimension between $x_k$ and $\bigcirc_s$ after applying Hanany--Witten transitions to move $x_1,\dots,x_w$ clockwise through $\bigcirc_1,\dots,\bigcirc_{n}$ $t-1$-times and further move $x_1,\dots,x_k$ through $\bigcirc_1,\dots,\bigcirc_s$. If $sk=0$, we define:
\begin{equation*}
    \begin{aligned}
        &\cd^1_{s,0}=\cD^1_{s,0}=v_s & &\cd^1_{0,k}=\cD^1_{0,k}=v_{-k}& &\\
        &\cd^t_{0,k} = \cd^{t-1}_{n,k} & & \cd^t_{s,0} = \cd^{t-1}_{s,w}& &  \text{ if }t\geq1.
    \end{aligned}
\end{equation*}
The previous relations are the same satisfied by $\cD^t_{s,k}$, see Remark \ref{R:interpretationinequalities}.

By replacing ``clockwise" with ``anticlockwise", we can define $\ad^t_{s,k}$ for $t\in\NN^+$, $s \in \{0,1,\dots,n\}$ and $k\in\{0,1,\dots,w\}$.

The claim follows from the following Lemma.
\begin{Lemma}\label{Lemma-interpretationsusyinequalities}
    Under the previous assumption, $\cd^t_{s,k} = \cD^t_{s,k}$ and $\ad^t_{s,k} = \aD^t_{s,k}$ for all $t$, $s$, and $k$.
\end{Lemma}
\begin{proof} We will prove by induction on $t$ that $\cd^t_{s,k} = \cD^t_{s,k}$ for all $t$, $s$, and $k$. The statement for the anticlockwise inequalities follows in the same way, so we omit the proof.

Suppose $t=1$ and let us proceed by induction on the product $sk$. If $sk=0$ it follows by definition. Let us fix $s\in\{1,\dots,n\}$ and $k\in\{1,\dots,w\}$ and let us suppose that $\cd^1_{s',k'}=\cD^1_{s',k'}$ for $s'k'<sk$.
By definition, if we have moved $x_1,\dots,x_{k-1}$ through $\bigcirc_1,\dots,\bigcirc_s$ and $x_k$ through $\bigcirc_1,\dots,\bigcirc_{s-1}$, then we have the following local configuration:
\begin{equation*}
    \input{Fig/step2picture}
\end{equation*}

Therefore, by the definition of Hanany--Witten transitions and the induction hypothesis, we have:
\begin{equation*}
       \begin{aligned} \cd^1_{s,k}&=\cd^1_{s,k-1}+\cd^1_{s-1,k}-\cd^1_{s-1,k-1}+1\\
      &=\cD^1_{s,k-1}+\cD^1_{s-1,k}-\cD^1_{s-1,k-1}+1\\
      &=sk+v_s+v_{-k}-v_0=\cD^1_{s,k}.
       \end{aligned}
\end{equation*}
Let us fix $t>1$ and suppose the claim is true for any $t'\leq t-1$. After moving $x_1,\dots,x_w$ clockwise $t-1$ times thorough $\bigcirc_1,\dots,\bigcirc_{n}$ we obtain the following separated diagram:
\begin{center}
     \input{Fig/Figure-_t-1_rotations}
\end{center} 

We notice that, from the induction hypothesis, we have $\cd^{t-1}_{s,w}=\cD^{t-1}_{s,w}$ and $\cd^{t-1}_{n,k}=\cD^{t-1}_{n,k}$ for every $s\in\{1,\dots,n\}$ and $k\in\{0,\dots,w\}$. 

Hence, from the case $t=1$, given $s\in\{1,\dots,n\}$ and $k\in\{1,\dots,w\}$, we have:
\begin{equation*}
\begin{aligned}
    \cd^t_{s,k}&=sk+ \cd^{t-1}_{s,w}+\cd^{t-1}_{n,k}-\cd^{t-1}_{n,0}\\
&=sk+ \cD^{t-1}_{s,w}+\cD^{t-1}_{n,k}-\cD^{t-1}_{n,0}=\cD^t_{s,k}
\end{aligned}
\end{equation*}
where the last identity follows from a direct computation.
\end{proof}

\subsection{Proof of Step 3}\label{s.step3}
We prove that a separated bow diagram satisfying the supersymmetry inequalities \ref{susyinequalities} is supersymmetric (Definition \ref{definitionsupersymmetry}).

The strategy is to apply Hanany--Witten transitions to obtain a separated diagram that satisfies additional constraints, enabling the application of `supersymmetric decrements' (Remark \ref{Rem-susyincrements}) to reduce the problem to the finite type A case.


\begin{Lemma}\label{Lemma.reducing}
 Any separated bow diagram of affine type A is Hanany--Witten equivalent to a bow diagram $(B,\Lambda,\vv)$:
 \begin{equation*}
     \input{Fig/StypeAffine3}
 \end{equation*}
such that $w> v_0-v_{-w}\geq0$.
\end{Lemma}
\begin{proof}Given a separated bow diagram as in Figure \ref{StypeAffine}, by moving $\bigcirc_1$ anticlockwise through all the x-points, we obtain:
\begin{equation*}
    \input{Fig/StypeAffine4}
\end{equation*}
where $v_{-w}^{new}=\cD^1_{1,w}$ follows from the proof of step 2 (see \S\ref{s.step2} or Remark \ref{R:interpretationinequalities}). Then, $v_0^{new}-v_{-w}^{new}=v_0-v_{-w}-w$. Similarly, if we move $\bigcirc_n$ clockwise through all the x-points, we can increase $v_0-v_{-w}$ by $w$.
\end{proof}
\begin{Remark}
We notice that in Lemma \ref{Lemma.reducing} we did not assume supersymmetry inequalities. An alternative proof of this lemma can be obtained using Remark \ref{Rem:HWandN} instead of $\cD^t_{s,k}$ and $\aD^t_{s,k}$. Also, by inverting the role of arrows and x-points we could change the difference $v_0-v_{-w}$ by $\pm n$.
\end{Remark}



\begin{Proposition}\label{Prop.reducing}
    Let us consider a separated bow diagram $(B,\Lambda,\vv)$ of affine type A, as in Figure \ref{StypeAffine}, satisfying supersymmetry inequalities and such that $w> v_0-v_{n} \geq0$. Then, given $a\leq v_s$ for $s\in\{1\dots,n\}$, the bow subdiagram:
    \begin{equation*}
        \input{Fig/redStypeAffine}
    \end{equation*}
satisfies supersymmetry inequalities. Moreover, if the new bow diagram is supersymmetric, then $(B,\Lambda,\vv)$ must also be supersymmetric.
\end{Proposition}
\begin{proof}
The inequalities follow from a direct computation. Let us denote the left-hand side of the supersymmetry inequalities (\ref{susyinequalities}) for the new diagram by $\bigl(\cD^t_{s,k}\bigr)'$ and $\bigl(\aD^t_{n+1-s,w+1-k}\bigr)'$.
        
Let us check the clockwise inequalities. For $k\neq0,w$; we have $\bigl(\cD^t_{s,k}\bigr)'=\bigl(\cD^t_{s,k}\bigr)\geq0$. For $k=w$, we have:
\begin{equation*}
        \bigl(\cD^t_{s,w}\bigr)'=tsw+\frac{(t-1)t}{2}wn+v_s-a-t(v_0-v_{-w})
\end{equation*}
and the claim immediately follows from $w\geq v_0-v_w$ and $v_s\geq a$ for every $s$. If $k=0$ and $t>1$ we are in one of the previous case (Remark \ref{R:interpretationinequalities}). If $k=0$ and $t=1$ we have $\bigl(\cD^1_{s,0}\bigr)'=v_s-a\geq0$.

Let us check the anticlockwise inequalities. If $k\neq0,w$; then $$\bigl(\aD^t_{n+1-s,w+1-k}\bigr)'=\bigl(\aD^t_{n+1-s,w+1-k}\bigr)\geq0.$$ If $k=w$, the inequalities follow from $v_0-v_{-w}\geq0$ and $v_s\geq a$ for every $s$. The case $k=0$ follows as in the clockwise case. 

The second part of the theorem follows from Remark \ref{Rem-susyincrements} by noticing that the original bow diagram is obtained from the second one by supersymmetric increments.
\end{proof}

To recap, let $(B,\Lambda,\vv)$ be a separated bow diagram of affine type A. We can suppose $0\leq v_0-v_{-w}\leq w$ (Lemma \ref{Lemma.reducing}). Suppose $(B,\Lambda,\vv)$ satisfies supersymmetry inequalities. Then, we can reduce the problem to study the supersymmetry of a bow subdiagram $(B,\Lambda,\vv')$ as in Proposition \ref{Prop.reducing}. 
In particular, if we set $a=\min\{v_1,\dots,v_n\}$, then $(B,\Lambda,\vv')$ is of finite type A and it satisfies supersymmetry inequalities. However, generally, as a finite type A diagram it is not separated. However, we have proved that supersymmetry inequalities are equivalent to require that every sequence of Hanany--Witten transitions cannot generate negative dimensions (see \S\ref{s.step2}). Therefore, $(B,\Lambda,\vv')$ is Hanany--Witten equivalent to a separated bow diagram of finite type A satisfying supersymmetry inequalities.

This reduces the problem to the case of a separated bow diagram of finite type A, which is addressed in the following proposition.
\begin{Proposition}\label{Prop-finitetypeAsupersymmetry}
    Let us consider separated bow diagram of finite type A as in Figure \ref{StypeA}. Suppose it satisfies clockwise\footnote{This choice depends on the convention we adopted to index separated finite type A bow diagrams.} supersymmetry inequalities for $t=1$, that is:
\begin{equation}
    \cD^1_{s,k}=sk+v_s+v_{-k}-v_0\geq0
\end{equation}
for every $s\in\{0,1,\dots,n\}$ and $k\in\{0,1,\dots,w\}$. Then, it is supersymmetric.
\end{Proposition}
\begin{proof}
We proceed by induction on the number of arrows\footnote{Here, the role of arrows (i.e. NS5-branes) and x-points (i.e. D5-branes) is completely interchangeable.} $n$.

If $n=0$, then supersymmetry inequalities becomes $v_{-k}\geq0$ for every $k$. In this case, these inequalities are equivalent to supersymmetry by construction.

If $n=1$, then
\begin{equation*}
    \input{Fig/N=0}
\end{equation*}
and $0\leq\cD^1_{1,k}=k+v_{-k}-v_0$. That is, $v_{-k}\geq v_0-k$ for every $k$. Let us construct a brane diagram generating $(B,\Lambda,\vv)$. By the supersymmetry inequalities, we can stretch one D3-brane with endpoints on $\bigcirc_1$ and $x_j$ for every $j\leq v_0\leq w$. The remaining D3-branes we left to arrange are only related with segments between $x_1$ and $x_w$. Therefore, they must necessarily be unfixed and so their arrangement is irrelevant to supersymmetry. That is, an associated brane diagram is given by:
\begin{center}
      \input{Fig/Figure-ArrangementN=0}
\end{center}

Let us suppose the claim is true for $n'<n$ with $n>1$. We construct now a supersymmetric brane diagram for such a bow diagram. We define
    \begin{equation}
        f\coloneqq \operatorname{max}\Bigl\{ k \in \{0,1,\dots,w\} \ \mid \ v_{-j} \geq k-j \ \ \ \ \ \forall j\in\{0,\dots,k\}\Bigr\}.
    \end{equation}
    That is, $f$ is the maximum index for which we can stretch $f$ fixed D3-branes between $\bigcirc_1$ and $x_{f}$ without breaking supersymmetry. Note, $f=0$ if and only if $v_0=0$. In such a case the claim is trivially true. Let us suppose $v_0>0$. Then, by definition of $f$, we start to construct the brane diagram by stretching a D3-brane with endpoints on $\bigcirc_1$ and $x_k$ for every $k\in\{1,\dots,f\}$, that is:
\begin{center}
    \input{Fig/Figure-Definitionoff}
\end{center}

We still need to arrange $v_0-f$ fixed D3-branes. Suppose $v_0-f>0$. Since we cannot attach more fixed D3-branes to $\bigcirc_1$ we will proceed to remove it. We have to fix $v_0-f$ D3-branes passing between $\bigcirc_1$ and $x_1$ but with no endpoints on $\bigcirc_1$. So, the first thing we need is $v_1 \geq v_0-f$. From supersymmetry inequalities, if $v_1<v_0-f$, then for every $j$ we have
\begin{equation*}
	 v_{-j} \geq v_0-v_1-j \geq f+1-j,
\end{equation*}
which contradicts the maximality of $f$. 

Furthermore, by definition of $f$, there must exists $\Tilde{j}\in\{0,1,\dots,f\}$ such that $v_{-\Tilde{j}}=f-\Tilde{j}$. Therefore, in any associated brane diagram having a fixed D3-brane between $\bigcirc_1$ and $x_k$ for every $k\in\{1,\dots,f\}$, we cannot have another fixed D3-brane with an endpoint on $x_k$ for some $k>\Tilde{j}$. 

We also notice that among the $v_1$ D3-branes we have to stretch between $\bigcirc_2$ and $\bigcirc_1$, only $v_0-f$ has to be fixed. Therefore, by attaching $v_1-(v_0-f)$ D3-branes between $\bigcirc_2$ and $\bigcirc_1$, we reduce the problem to show the supersymmetry for the bow diagram:
\begin{center}
    \input{Fig/InductionfinitetypeA}
\end{center}
That is, a separated finite type A diagram, with $n-1$ arrows, $\Tilde{j}$ x-points and dimension vector given by:
\begin{equation*}
\begin{aligned}
    &v'_s=v_{s+1} & \text{ if }& s\in\{1,\dots,n-1\},\\
    &v'_{-k}=v_{-k}-(f-k)  &\text{ if }& k\in\{0,\dots,\Tilde{j}\}.
\end{aligned}
\end{equation*}
By the induction hypothesis, it is enough to show that it satisfies supersymmetry inequalities. Let us compute them. Given $s\in\{1,\dots,n-1\}$ and $k\in\{1,\dots,\Tilde{j}\}$:
\begin{equation*}
        (\cD^1_{s,k})^{new}=sk+v_{s+1}+v_{-k}-(f-k)-(v_0-f)=D^1_{s+1,k}\geq 0.
    \end{equation*}
\end{proof}
\begin{Remark}
     The argument used for the case $n=1$ in the proof of Proposition \ref{Prop-finitetypeAsupersymmetry} applies to any $n\in\NN^+$ when $v_1=0$.
\end{Remark}
\begin{Remark}\label{rem:susydecrements}
     Proposition \ref{Prop.reducing} can be interpreted as a `supersymmetric decrement'. Indeed, since we have proved $(c)\Leftrightarrow(d)\Leftrightarrow(e)$, it is saying that under certain conditions, we can safely reduce the dimension vector without breaking supersymmetry.
\end{Remark}

\subsection{Proof of Step 4} 
We prove that if a bow diagram of affine type A $(B,\Lambda,\vv)$ is supersymmetric then $ \mu_{(B,\Lambda,\vv)}^{-1}(0) \neq \varnothing $.

Proceeding as in the proof of Step 3 (see \S\ref{s.step3}), up to apply a sequence of Hanany--Witten transitions, we can suppose that $(B,\Lambda,\vv)$ is constructible by applying supersymmetric increments between arrows to a supersymmetric finite type A bow diagram. 

Since Hanany--Witten transitions correspond to isomorphisms of bow varieties (Theorem \ref{HWtransitionsforbow}), it suffices to prove that supersymmetric increments preserve the existence of a solution for $\mu=0$ (Theorem \ref{supersymmetricincrementsforbows}) and that such a solution exists for supersymmetry bow diagrams of finite type A (Theorem \ref{Theorem-susyfinitetypeAimpliesnonemptiness}).

\begin{Theorem}\label{supersymmetricincrementsforbows}
    Let us fix a bow diagram $(B,\Lambda,\vv)$ such that $\mu^{-1}(\lambda) \neq \varnothing$ for some $\lambda\in\CC^\II$.
\begin{enumerate}[label=(\roman*)]
    \item Suppose $\lambda=0$. If $(B,\Lambda,\vv^{new})$ is obtained from $(B,\Lambda,\vv)$ via supersymmetric increments between arrows, then $\mu^{-1}_{(B,\Lambda,\vv^{new})}(0)\neq\varnothing$.
    \item Suppose $x_1,x_2\in\Lambda$ are two consecutive x-points on the same wavy line and ordered according the wavy line's orientation. Let $\zeta=\zeta_{x_1}^+=\zeta_{x_2}^-$. If $\vv^{new} \in \NN^{\IS}$ such that $v^{new}_{\zeta}-v_{\zeta}\in\NN$ and $v^{new}_{\zeta'}=v_{\zeta'}$ for $\zeta'\neq\zeta$; then $\mu^{-1}_{(B,\Lambda,\vv^{new})}(\lambda)\neq\varnothing$.
\end{enumerate}
\end{Theorem}
\begin{proof}
We prove \textit{(i)}. It is enough to prove the case in which there exists a wavy line $\sigma\in\II$ such that
\begin{equation*}
    v^{new}_\zeta=\begin{cases}
        v_{\zeta} & \text{ if }\zeta\not\subset\sigma,\\
        v_{\zeta}+1 & \text{ if }\zeta\subset\sigma.
    \end{cases}
\end{equation*}
Let us fix $m=(A,B^-,B^+,a,b,C,D)\in\mu^{-1}_{(B,\Lambda,\vv)}(0)$. Then we define $m^{new}\in\mu^{-1}_{(B,\Lambda,\vv^{new})}(0)$ in the following way.

Let $x \in \Lambda$. If $x\notin\sigma$, $$(A^{new}_x,(B^-_x)^{new},(B^+_x)^{new},a_x^{new},b_x^{new})=(A_x,B^-_x,B^+_x,a_x,b_x).$$ If $x \in \sigma$, then
\begin{equation*}
    \input{Fig/Figure-1susyincrements}
\end{equation*}
Let $e\in\EE$, then we define $(C^{new}_e,D^{new}_e)$ by extending $(C_e,D_e)$ as zero on the added copy of $\mathbb{C}$. It is easy to see that $m^{new}\in\mu_{(B,\Lambda,\vv^{new})}^{-1}(0)$ concluding the proof of \textit{(i)}.

To prove \textit{(ii)}, it suffices to consider the case where $v^{new}_\zeta-v_\zeta=1$. Let us fix $m\in\mu_{(B,\Lambda,\vv)}^{-1}(0)$. Then we define $m^{new}$ in the following way.
For every arrow $e$, $$(C^{new}_e,D^{new}_e)=(C_e,D_e).$$ For every x-point $x \neq x_1,x_2$, we define 
$$(A^{new},(B^{new})^\pm,a^{new},b^{new})=(A,B^\pm,a,b).$$
Finally, given $c\in\CC$ such that $(B_{x_1}^--c)$, $(B_\zeta-c)$ and $(B_{x_2}^+-c)$ are invertible, we define:
\begin{equation*}
    \input{Fig/Figure-2susyincrements}
\end{equation*}
where $A_{x_1}^{new}$ and $A_{x_2}^{new}$ are defined by condition (\ref{condition-a}): $B^+A-AB^-+ab=0$. That is:
\begin{equation*}
    A_{x_1}^{new}=\begin{bmatrix}
        A_{x_1} \\
        +b_{x_1}(B_{x_1}^--c)^{-1}
    \end{bmatrix},
\ \ \ \ \      A_{x_2}^{new}=\begin{bmatrix}
         A_{x_2} & -(B_{x_2}^+-c)^{-1}a_{x_2}
     \end{bmatrix}.
\end{equation*}
To conclude that $m^{new} \in \mu_{(B,\Lambda,\vv^{new})}^{-1}(0)$, the only non-trivial conditions we need to verify are the condition (\ref{condition-S1}) for $\Bigl(A^{new}_{x_2}, (B^{new})^\pm_{x_2}, a^{new}_{x_2}, b^{new}_{x_2}\Bigr)$ and the condition (\ref{condition-S2}) for $\Bigl(A^{new}_{x_1},(B^{new})^\pm_{x_1}, a^{new}_{x_1}, b^{new}_{x_1}\Bigr)$.

We begin with condition (\ref{condition-S1}). 
Consider $S\subset\Ker A^{new}_{x_2} \cap \Ker b^{new}_{x_2}$, $B^{new}_\zeta$-invariant. By construction,
\begin{equation}\label{(S1)2}
    \begin{aligned}
    \Ker A^{new}_{x_{2}}\cap\Ker b^{new}_{x_2}= &\Bigl\{\begin{bmatrix} s \\ -b_{x_2}s \end{bmatrix} \ \vert \ A_{x_2}s=-(B_{x_2}^+-c)^{-1}a_{x_2}b_{x_2}s\Bigr\}
\end{aligned}
\end{equation}

Let us consider the projection:
\begin{equation*}\bar{S}=\Bigl\{s\in\CC^{v_\zeta} \ \vert \ \begin{bmatrix} s \\ -b_{x_2}s \end{bmatrix} \in S\Bigr\}\subset\CC^{v_\zeta}\oplus\{0\}\subset\CC^{v_\zeta}\oplus\CC.\end{equation*} Then:
\begin{enumerate}
    \item $(B_\zeta-c)\bar{S}\subset \Ker A_{x_2}\cap \Ker b_{x_2}$. Indeed, since $S$ is $B^{new}_\zeta$-invariant, from (\ref{(S1)2}) it suffices to notice that for every $s\in\bar{S}$
\begin{equation*}
    \begin{bmatrix}
        B_\zeta s \\
        -cb_{x_2}s
    \end{bmatrix}-c    \begin{bmatrix}
        s \\
        -b_{x_2}s
    \end{bmatrix}=    \begin{bmatrix}
        (B_\zeta-c) s \\
        0
    \end{bmatrix}\in S.
\end{equation*}
    \item $(B_\zeta-c)\bar{S}$ is $B_\zeta$-invariant. Indeed, since $S$ is $B^{new}_\zeta$-invariant, $\bar{S}$ is $B_\zeta$-invariant.
\end{enumerate}
Therefore, since $(A_{x_2},B_{x_2}^\pm,a_{x_2},b_{x_2})$ satisfies condition (\ref{condition-S1}), $(B_\zeta-c)\bar{S}=0$. By assumption, $(B_\zeta-c)$ is invertible and therefore $\bar{S}=0$. It follows by (\ref{(S1)2}) that $S=0$.

Let us prove condition (\ref{condition-S2}) for $\Bigl(A^{new}_{x_1},(B^{new})^\pm_{x_1}, a^{new}_{x_1}, b^{new}_{x_1}\Bigr)$ by using a dual argument. Let us consider $T\subset\CC^{v_\zeta}\oplus\CC$ such that $\Ima A^{new}_{x_1} + \Ima a_{x_1}^{new}\subset T$ and $T$ is $B^{new}_\zeta$-invariant. By construction
\begin{equation}\label{(S2)1}
    \Ima A^{new}_{x_1} + \Ima a_{x_1}^{new}=\Bigl\{ \begin{bmatrix}
        A_{x_1}s+a_{x_1}\lambda \\
        b_{x_1}\bigl(B_{x_1}^--c\bigr)^{-1}s+\lambda
    \end{bmatrix} \ \vert \ s\in\CC^{v_{x_1}^-},\ \lambda\in\CC \Bigr\}.
\end{equation}
Let us denote by $T_0\oplus0$ the intersection $T\cap \bigl(\CC^{v_\zeta}\oplus0\bigr)$. Then:
\begin{enumerate}
    \item $\Ima A_{x_1} + \Ima a_{x_1} \subset \bigl(B_\zeta-c\bigr)^{-1}T_0$. 
    Indeed, since $T$ is $B^{new}_\zeta$-stable, from (\ref{(S2)1}) it suffices to notice that for every $s\in\CC^{v_{x_1}^-}$ and $\lambda\in\CC$, \begin{equation*}
    \begin{gathered}
        \begin{bmatrix}
        B_\zeta\bigl(A_{x_1}s+a_{x_1}\lambda\bigr) \\
        c(b_{x_1}\bigl(B_{x_1}^--c\bigr)^{-1}s+\lambda)
    \end{bmatrix}-c\begin{bmatrix}
        A_{x_1}s+a_{x_1}\lambda \\
        b_{x_1}\bigl(B_{x_1}^--c\bigr)^{-1}s+\lambda
    \end{bmatrix}\\
    =\begin{bmatrix}
     \bigl(B_\zeta-c\bigr)\bigl(A_{x_1}s+a_{x_1}\lambda\bigr) \\
       0
    \end{bmatrix}\in T
    \end{gathered}
    \end{equation*}
\item $\bigl(B_\zeta-c\bigr)^{-1}T_0$ is $B_\zeta$-invariant. Indeed, since $T$ is $B^{new}_\zeta$-invariant, $T_0$ is $B_\zeta$-invariant.
\end{enumerate}
Since $(A_{x_1},B_{x_1}^\pm,a_{x_1},b_{x_1})$ satisfies condition (\ref{condition-S2}), we have $T_0=\CC^{v_\zeta}$ and therefore $\CC^{v_\zeta}\oplus0\subset T$. By definition of $a^{new}_{x_1}$ we can deduce $T=\CC^{v_\zeta}\oplus\CC$.
\end{proof}
\begin{Corollary}\label{corollarysupersymmetricincrementsforbow}
    Let $(B,\Lambda,\vv)$ be a bow diagram and $(B,\Lambda,\vv^{new})$ a bow diagram obtained from $(B,\Lambda,\vv)$ via supersymmetric increments between x-points. If there exists $\lambda\in\CC^{\II}$ such that $\mu^{-1}_{(B,\Lambda,\vv)}(\lambda)\neq\varnothing$, then $\mu^{-1}_{(B,\Lambda,\vv^{new})}(\lambda)\neq\varnothing$.
\end{Corollary}
\begin{proof}
We notice that $x_1$ and $x_2$ divide the bow diagram in two arcs, which we denote by $L$ and $R$. Without loss of generality, we can assume that $\vv^{new}$ and $\vv$ coincide on every segment in $L$. We can apply Hanany--Witten transitions to move x-points along the part $R$ of the bow diagram in order to obtain an equivalent bow diagram having no arrows between $x_1$ and $x_2$. Then, increasing the dimension vectors between $x_1$ and $x_2$ in the first bow diagrams corresponds to increase the dimension vectors between $x_1$ and $x_2$ in the second bow diagram. The claim follows from Theorem \ref{supersymmetricincrementsforbows} and Theorem \ref{HWtransitionsforbow}.
\end{proof}
\begin{Remark}\label{Rem.susyincrementsforbowdiagramsingeneraltype}
We notice that Theorem \ref{supersymmetricincrementsforbows} applies to bow varieties with every underlying quiver, although we do not currently have an interpretation in terms of brane diagrams for the general case. For a quiver description of these varieties, see \cite{gaibisso2024quiver}.
\end{Remark}
\begin{Remark}
 Point $(i)$ of Theorem \ref{supersymmetricincrementsforbows} does not generalize to  $\lambda\neq0$. Indeed, due to the correspondence between cobalanced bow varieties and Nakajima quiver varieties \cite{cherkis2011instantons}, \cite{NT17}, this would imply that every Nakajima quiver variety with a trivial stability parameter is non-empty, which is false.
\end{Remark}
Thus, to conclude, we need to prove the claim for finite type A bow diagrams.
\begin{Theorem}\label{Theorem-susyfinitetypeAimpliesnonemptiness}
    If $(B,\Lambda,\vv)$ is a finite type A supersymmetric bow diagram, then $\mu^{-1}(0) \neq \varnothing$.
\end{Theorem}
\begin{proof}
Without loss of generality (Theorem \ref{HWtransitionsforbow}), we can assume $(B,\Lambda,\vv)$ is separated. Let $\mathcal{B}$ denote a supersymmetric brane diagram associated with this bow diagram. Let us denote by $l_i$ the number of fixed D3-branes attached to $x_i$. Without loss of generality, we can assume that $l_1 \geq l_2 \geq \dots \geq l_w$ and that the $l_i$ NS5-branes connected to $x_i$ are the ones corresponding to $\bigcirc_1,\dots,\bigcirc_{l_i}$. Note that $v_0 = l_1 + \cdots + l_w$. 

Let $\mathcal{B}'$ be the brane diagram obtained by removing all unfixed D3-branes from $\mathcal{B}$. We denote by $(B,\Lambda,\vv')$ the bow diagram associated with $\mathcal{B}'$. Then, by performing Hanany--Witten transitions, we can eliminate all the D3-branes in $\mathcal{B}'$. In fact, due to the ordering assumption on $l_1, \dots, l_w$, this can be achieved by sequentially moving $x_i$ through $\bigcirc_1, \dots, \bigcirc_{l_i}$ for each $i \in {1, \dots, w}$. Consequently, $(B, \Lambda, \vv')$ is Hanany--Witten equivalent to a bow diagram $(B, \Lambda', 0)$ that admits a solution for $\mu_{(B, \Lambda', 0)} = 0$. From the interpretation of Hanany--Witten transitions in terms of bow varieties, (Theorem \ref{HWtransitionsforbow}), it follows that $\mu^{-1}_{(B, \Lambda, \vv')}(0)\neq\varnothing$. 

Since $(B, \Lambda, \vv)$ is obtained from $(B,\Lambda,\vv')$ via supersymmetric increments, it follows from Theorem \ref{supersymmetricincrementsforbows} that $\mu^{-1}_{(B, \Lambda, \vv)}(0)\neq\varnothing$.
\end{proof}

%% file: 004supersymmetryinequalitiesandthestratumcondition.tex
\subsection{Proof of Step 5}\label{s.step5}
We prove that a bow diagram is Hanany–Witten equivalent to a separated diagram satisfying the supersymmetry inequalities (Definition \ref{HWinequalities}) if and only if it satisfies the stratum condition (Definition \ref{def:stratumcondition}).

While we present an algebraic proof, we also provide combinatorial motivation for our steps using brane systems. These arguments could, in principle, be extended to give a direct proof of the equivalence between supersymmetry and the stratum condition, that is $(e)\Leftrightarrow(f)$. However, we do not develop such a proof in detail.

The main idea behind the proof is as follows: We anticipate that these conditions are equivalent to supersymmetry, and we have already shown that determining the supersymmetry of affine type A diagrams reduces to the finite type A case. This observation suggests that a similar reduction strategy can be applied here, which underpins the approach we take in the proof.
\begin{Proposition}\label{prop.inequalitiesstratum}
    A separated bow diagram of affine type A satisfies supersymmetry inequalities if and only if it satisfies the stratum condition.
\end{Proposition}
\begin{proof}
    Let us consider a separated bow diagram $(B,\Lambda,\vv)$ as in Figure \ref{StypeAffine} and let us set $\prescript{t}{}{\lambda}_s=v_{s-1}-v_s$ an $\mu_i=v_{-(i-1)}-v_{-i}$ for every $s\in\{1,\dots,n\}$ and $i\in\{1,\dots,w\}$. That is, $(B,\Lambda,\vv)=\mathcal{B}(\prescript{t}{}{\lambda}_s,\mu_i,v_n):$
    
    \begin{equation}\label{diagram1}
        \input{Fig/separatedsusyinequalitiesAffine}
    \end{equation}

Up to considering a Hanany--Witten equivalent bow diagram, we can suppose $w>v_{0}-v_{-w}\geq0$; that is, $w>\vert\prescript{t}{}{\lambda}\vert=\vert\mu\vert\geq0$ (see Lemma \ref{Lemma.reducing}). Let $v_s=\min\{v_0,v_1,\dots,v_{n-1}\}$ and consider the subdiagram obtained from (\ref{diagram1}) by subtracting $v_s$ from $v_0,v_1,\dots,v_n$, that is:
\begin{equation}\label{diagram2}
    \input{Fig/separatedsusyinequalitiesAffine2}
\end{equation}
The bow diagram (\ref{diagram1}) satisfies the supersymmetry inequalities if and only if the bow diagram (\ref{diagram2}) does. The ``only if" direction was established in Proposition \ref{Prop.reducing}, while the ``if" direction follows either from a direct computation or from Theorem \ref{maintheorem} and Remark \ref{Rem-susyincrements} by noticing that (\ref{diagram1}) is obtained from (\ref{diagram2}) via supersymmetric increments.

Furthermore, the bow diagram (\ref{diagram1}) satisfies the stratum condition if and only if the bow diagram (\ref{diagram2}) does (see Proposition \ref{Lemma.reducingstratum}).

Thus, we have reduced the problem to the finite type A case, which is proved in Proposition \ref{Prop.inequalitiesstratumfinite}.
\end{proof}

\begin{Proposition}\label{Lemma.reducingstratum}
  Under the previous assumptions, the bow diagram in (\ref{diagram1}) satisfies the stratum condition if and only if the bow diagram (\ref{diagram2}) satisfies the stratum condition.
\end{Proposition}
\begin{proof}
Let us notice that both the diagrams (\ref{diagram1}) and (\ref{diagram2}) satisfy conditions in Lemma \ref{Affinestratumcondition} and therefore we reduce the stratum condition to the existence of certain $(\Liegl_w)_{\text{aff}}$-weights. 

Suppose to have a $(\Liegl_w)_{\text{aff}}$-weight $\hat{\kappa}$ as in the stratum condition in Lemma \ref{Affinestratumcondition} for the bow subdiagram (\ref{diagram2}). We claim that $\kappa=\hat{\kappa}+v_s\delta$ is a weight for (\ref{diagram1}) as in Lemma \ref{Affinestratumcondition}. The only non trivial inequality we have to check is $\kappa\geq\mu$ as $(\Liesl_w)_{\text{aff}}$-weights using (\ref{eq.diagramstoweights}). It follows by noticing that, as $(\Liesl_w)_{\text{aff}}$-weights,  $\kappa-v_s\delta=\hat{\kappa}\geq\hat{\mu}=\mu+v_s\alpha_0-v_s\delta$, where $\hat{\mu}$ is the weight corresponding to $[\mu_1-v_s,\mu_2,\dots,\mu_{w-1},\mu_{w}+v_s]$ with $\langle\hat{\mu},d\rangle=0$. 

Conversely, let us suppose to have a weight $\kappa$ for the former diagram (\ref{diagram1}). As above, $\kappa-v_s\delta=\hat{\kappa}$ is a well-defined weight for the reduced diagram (\ref{diagram2}) and we only need to check $\hat{\kappa}\geq \mu+v_s\alpha_0-v_s\delta$, that is $\kappa\geq\mu+v_s\alpha_0$. This immediately follows from our assumption $w>\vert\prescript{t}{}{\lambda}\vert\geq0$. Indeed, we know that $\kappa-\mu=\sum\limits_{i=0}^{w-1}\hat{v}_i\alpha_i$ where $\hat{v}_i$ are the dimensions in the Hanany--Witten equivalent balanced bow diagram (\ref{balancediagram}).
Then, as shown in Example \ref{exseparatedtobalance}, $\hat{v}_0$ coincide with one of the dimensions adjacent to an arrow and $v_s$ has been taken to be the minimum among those dimensions.
\end{proof}

\begin{Remark}
Proposition \ref{Lemma.reducingstratum} and Proposition \ref{Prop.reducing} play analogous roles: each reduces the task of verifying the corresponding condition to finite type. Also, let us notice that in the proofs of both the propositions, to go from the bow subdiagram to the original one we are applying some supersymmetric increments which always preserve supersymmetry. Indeed, we do not use the assumption on the charges (or equivalently on the difference between the dimensions $v_0,v_n=v_{-w}$). However, as for supersymmetry inequalities, for the reverse arrows we need it. Indeed, in the most generality, there is no a safe way to reduce the dimension vector without breaking supersymmetry. See also Remark \ref{Rem-susyincrements} and Remark \ref{rem:susydecrements}.
\end{Remark}

\begin{Proposition}\label{Prop.inequalitiesstratumfinite}
Consider a separated bow diagram of finite type A:
\begin{equation}\label{separatedsusystrata}
    \input{Fig/separatedsusystratum}
\end{equation}
where $\prescript{t}{}{\lambda}_s=v_{s-1}-v_{s}$ and $\mu_k=v_{-(k-1)}-v_{-k}$ for every $s\in\{1,\dots,n\}$ and $k\in\{1,\dots,w\}$. Then, it satisfies supersymmetry inequalities if and only if it satisfies the stratum condition. That is, the following are equivalent:
\begin{enumerate}
    \item $\aD^1_{s,j}=sj+v_s+v_{-j}-v_0\geq0$ for every $s\in\{0,\dots,n\}$ and every $j\in\{0,\dots,w\}$.
    \item There exists a polynomial dominant $\Liegl_w$-weight $\kappa$ with level less than or equal to $n$ and $\vert \kappa \vert = \vert\mu\vert$, such that $\kappa\geq\mu$ and $\prescript{t}{}{\kappa}\leq\prescript{t}{}{\lambda}$.
\end{enumerate}
\end{Proposition}
\begin{proof}
 Suppose the bow diagram satisfies supersymmetry inequalities. Let us use the intuition coming from the relation between brane diagrams and bow diagrams to construct a weight $\kappa$ as in $(2)$. By Theorem \ref{maintheorem}, we know that it originates from a supersymmetric brane diagram with $v_0$ fixed D3-branes. Among these brane diagrams we can consider one obtained by attaching to $\bigcirc_1$ the maximum number of fixed D3-branes without breaking supersymmetry, to $\bigcirc_2$ we do the same with the remaining branes and so on. That is, the number of D3-branes attached to $\bigcirc_s$ is given by:
 \begin{equation*}
     \prescript{t}{}{\kappa}_s=\max\Bigl\{f=0,\dots,w \ \vert\  v_{-j}-\bigr(\sum\limits_{j=1}^{s-1}\prescript{t}{}{\kappa}_s-(s-1)j\bigl)\geq (f-j) \ \text{ for }j=0,\dots,w\Bigl\}
 \end{equation*}
 for every $s\in\{1,\dots,n\}$. Given such a brane diagram, we can consider the bow diagram corresponding to the brane subdiagram obtained by removing all the D3-branes stretched between NS5-branes (i.e. arrows). That is:
 \begin{equation}\label{finitebalancedseparatedsubdiagram}
     \input{Fig/finitebalancedseparatedsubdiagram}
 \end{equation}
  By construction we expect to have well-defined supersymmetric bow subdiagram. We also notice that, by moving $\bigcirc_s$ to the right of $x_{\prescript{t}{}{\kappa}_s}$ for every $s$, if well-defined, we obtain a balanced diagram. This can be proved, for example, using Remark \ref{Rem:HWandN}. In terms of brane diagrams, it follows by the interpretation of Hanany--Witten transitions for supersymmetric brane diagrams \S\ref{s.hw} and by noticing that our subdiagram originates from a brane diagram having all the D3-branes either fixed or stretched between two D5-branes (i.e. x-points).
  
Let us prove that the diagram in (\ref{finitebalancedseparatedsubdiagram}) satisfies the stratum condition. We now only use supersymmetry inequalities.

Let us start by noticing that, by maximality in the definition of $\prescript{t}{}{\kappa}_s$, for every $s\in\{1,\dots,n\}$, there exists $j_s\in\{0,\dots,\prescript{t}{}{\kappa}_s\}$ such that 
 \begin{equation}\label{eq.js}
     v_{-j_s}=\sum\limits_{i=1}^{s}\prescript{t}{}{\kappa}_i-sj_s
 \end{equation}
 First, we claim that the bow diagram (\ref{finitebalancedseparatedsubdiagram}) is well defined. That is: 
 \begin{equation*}
     \sum\limits_{s=1}^{n}\prescript{t}{}{\kappa}_s=v_0=\sum\limits_{i=1}^{w}\mu_i.
 \end{equation*}
 Indeed, by definition of $\prescript{t}{}{\kappa}_n$, we have $v_0\geq\sum\limits_{s=1}^{n}\prescript{t}{}{\kappa}_s$. 
On the other hand, from supersymmetry inequalities and (\ref{eq.js}) we have: 
\begin{equation*}
0\leq\aD^1_{n,j_n}=nj_n+v_{-j_n}-v_0=\sum\limits_{i=1}^{n}\prescript{t}{}{\kappa}_s-v_0.
\end{equation*}
Then, let us show that $w\geq\prescript{t}{}{\kappa}_1\geq\cdots\geq\prescript{t}{}{\kappa}_n\geq0$.  By definition:
\begin{equation*}
    0=v_{-w}\geq\prescript{t}{}{\kappa}_1-w
\end{equation*}
By (\ref{eq.js}) and definition of $\prescript{t}{}{\kappa}_{s+1}$:
\begin{equation*}
   \prescript{t}{}{\kappa}_s\geq j_s=j_s+\bigl(v_{-j_s}-\sum\limits_{i=1}^{s}\prescript{t}{}{\kappa}_i-sj_s\bigr)\geq \prescript{t}{}{\kappa}_{s+1}.
\end{equation*}
This is equivalent to the existence of a dominant $\Liegl_w$-weight $\kappa$ whose charge is $\vert\mu\vert$ and the level $\kappa_1$ is at most $n$ (see \S\ref{s.balance}). Therefore, we conclude if we show that $\kappa\geq\mu$ and $\prescript{t}{}{\kappa}\geq\prescript{t}{}{\lambda}$.
$\\$
Claim: $\kappa\geq\mu$.

Since $\sum\limits_{i=1}^{w}\kappa_i=\sum\limits_{i=1}^{w}\mu_i$, the claim is equivalent to $\sum\limits_{i=j}^{w}\kappa_i\leq\sum\limits_{i=j}^{w}\mu_i=v_{-(j-1)}$ for every $j\in\{1,\dots,w\}$. We also have $\kappa_j=\Bigl\vert\Bigl\{s\in\{1,\dots,w\} \ \vert \ \prescript{t}{}{\kappa}_s\geq j\Bigr\}\Bigr\vert$. Then:
\begin{align*}
    &\sum\limits_{i=j}^{w}\mu_i=v_{-(j-1)}\geq\sum\limits_{s=1}^{\kappa_j}\prescript{t}{}{\kappa}_s-(j-1)\kappa_j\\
    &=\sum\limits_{i=j}^{w-1}i\bigl(\kappa_i-\kappa_{i+1}\bigr)+w\kappa_w-(j-1)\kappa_j
    =\sum\limits_{i=j}^{w}\kappa_i
\end{align*}
where the first inequalities comes from the definition of $\prescript{t}{}{\kappa}$, the second and third equalities are straightforward from the definition of the objects involved.
$\\$
Claim: $\prescript{t}{}{\kappa}\geq\prescript{t}{}{\lambda}$.

We have to prove: $\sum\limits_{s=1}^{d}\prescript{t}{}{\kappa}_s\geq\sum\limits_{s=1}^d\prescript{t}{}{\lambda}_s=v_0-v_d$ for every $d\in\{1,\dots,n\}$.  Suppose there exists $d$ such that the inequalities does not hold. Then, since $\aD^1_{d,j}\geq0$, we have for every $j\in\{1,\dots,w\}$:
\begin{equation*}
    v_0-\sum\limits_{s=1}^{d}\prescript{t}{}{\kappa}_s > v_d \geq v_0-dj-v_{-j}.
\end{equation*}
This is in contraddiction with the maximality of $\prescript{t}{}{\kappa}_d$.

Let us prove the opposite direction. Suppose there exists a weight $\kappa$ as in $(2)$. As we have seen in \S\ref{s.balance}, we can think of it as the existence of a subdiagram (\ref{finitebalancedseparatedsubdiagram}) which is supersymmetric and Hanany--Witten equivalent to a balanced bow diagram. We know that $\sum\limits_{s=1}^{d}\prescript{t}{}{\kappa}_s\geq\sum\limits_{s=1}^{d}\prescript{t}{}{\lambda}_s=v_0-v_d$. Since we know the subdiagram is supersymmetric and this means it satisfies supersymmetry inequalities, we have:
\begin{equation*}
    jd+v_{-j}+v_d-v_0\geq jd+v_{-j}-\sum\limits_{s=1}^{d}\prescript{t}{}{\kappa}_s\geq0
\end{equation*}
for every $d\in\{1,\dots,n\}$ and $j\in\{1,\dots,w\}$. The cases $dj=0$ follow by noticing that $v_d,v_j\geq0$ since we are assuming the existence of a bow subdiagram originating from a brane diagram.
\end{proof}
\begin{Remark}
    We notice that, at the end of the previous proof, we used that supersymmetry implies supersymmetry inequalities, which follows from the original work \cite{HW}. However, one could avoid using it (explicitly) and concluding only using the properties of the involved weights.
\end{Remark}
\begin{Remark}
    In the finite case, one can see that the weight $\kappa$ obtained from the supersymmetry inequalities is the smallest possible with respect to the dominance order. As we noticed in Remark \ref{Rem:stratumconditionexistencestratum}, these $\kappa$ are in correspondence with strata of the non-deformed bow variety $\M_{0,0}$. With this interpretation: in the proof of Proposition \ref{Prop.inequalitiesstratumfinite}, we have used supersymmetry inequalities to construct the minimal stratum of $\M_{0,0}$ in the finite type A case.
\end{Remark}

%% file: 004Algorithm.tex
\section{An algorithm for supersymmetry}\label{section4} 
In this section, we present a systematic method to detect supersymmetry in bow diagrams. Using the proof of Theorem \ref{maintheorem}, we reduce the task of arranging D3-branes into a supersymmetric configuration to verifying a finite number of inequalities, which can be seen as a numerical criterion for detecting supersymmetry. Additionally, if the bow diagram is supersymmetric, this algorithm can be used to construct an associated supersymmetric brane diagram and a (possibly non-unique) solution to $\mu=0$. See Section \ref{section5} for explicit examples.

The key idea is that every supersymmetric brane diagram can be obtained by applying supersymmetric increments and Hanany--Witten transitions to a brane diagram without D3-branes.

The reader does not need prior knowledge of the proof of the main theorem to follow this section, as all necessary results are explicitly provided.

\subsection{Determining supersymmetry for bow diagrams}\label{algorithmsupersymmetry1}
Let us fix a generic bow diagram of affine type A, denoted $(B,\Lambda,\vv)$. For the entire subsection, we will use notation in Figure \ref{StypeAffine}.  

Since the number of x-points and arrows is finite, we can apply a sequence of Hanany--Witten transitions to obtain an equivalent separated bow diagram. If, at any step of this process, a transition produces a negative dimension, then the diagram is not supersymmetric. This follows directly from the definition of supersymmetry and its invariance under Hanany--Witten transitions.

This reduces the problem to the case of a separated bow diagram $(B,\Lambda',\vv')$. By applying a finite number of Hanany--Witten transitions, we may assume that $0\leq v'_0 - v'_{-w} < w$ (see Lemma \ref{Lemma.reducing} or Remark \ref{Rem:HWandN}). As before, if any of these transitions generates a negative dimension, the diagram would not be supersymmetric.  

Given such a bow diagram, we define a bow subdiagram $(B,\Lambda'',\vv'')$ obtained by decreasing the entries of the dimension vectors adjacent to arrows, that is $v'_s$ for $s\in\{0,\dots\ n\}$, by $a=\min\{v_0,\dots,v_n \}$.

From Proposition \ref{Prop.reducing} and Theorem \ref{maintheorem}, $(B,\Lambda',\vv')$ is supersymmetric if and only if $(B,\Lambda'',\vv'')$ is supersymmetric. By definition, such a bow subdiagram is of finite type A and therefore we reduced to study supersymmetry of a finite type A bow diagram. Finally, from Proposition \ref{Prop-finitetypeAsupersymmetry} and Theorem \ref{maintheorem}, the supersymmetry of a separated finite type A diagram is equivalent to verifying a finite number of supersymmetry inequalities corresponding to checking that we can move every x-point through every arrow via Hanany--Witten transitions without generating negative dimensions.

\begin{Remark}
We do not yet have an optimal version of the previous algorithm. 

First, given a separated finite type A diagram as in Figure \ref{StypeA}, it is not necessary to check $nw$ inequalities to determine if it is supersymmetric. Indeed, if $$n'=\min\bigl(s\in\{1,\dots,n\} \ \vert \ v_{s}=0\bigr), \ \ w'=\min\bigl(k\in\{1,\dots,w\} \ \vert \ v_{-k}=0\bigr)$$ then it suffices to verify at most $\min(n',v_0)\min(w',v_0)$ inequalities, specifically, those corresponding to move $x_1,\dots,x_{w'}$ through $\bigcirc_1,\dots,\bigcirc_{n'}$.

Additionally, if we start the algorithm with a bow diagram $(B,\Lambda,\vv)$ of affine type A with a segment $\zeta$ such that $v_\zeta = 0$, then we are in the finite type A case, and the previous algorithm could become unnecessarily long. 

Another example in which the algorithm could not be optimal is given by bow diagrams that are Hanany--Witten equivalent to a balanced bow diagram. In this case, we can construct the corresponding $(\Liesl_w)_{\text{aff}}$-weights $\lambda,\mu$ and supersymmetry is equivalent to $\lambda\geq\mu$ in the dominance order. This, is equivalent to checking that the entries of the dimension vector of the equivalent balanced bow diagram are non-negative (see Section \ref{s.balance}). Similarly, we can handle the case of bow diagrams that are Hanany--Witten equivalent to a cobalanced bow diagram.

\end{Remark}
\begin{Remark}\label{Remark.b.c.d.}
In string theory, brane systems associated with quivers of type B, C, and D are well-defined (cf. \cite{GK99}, \cite{bourget2023branes}). The construction is more complex because it requires the introduction of new affine subspaces, called orientifolds, which have special properties depending on the case. Nevertheless, even in this context, there exists a notion of supersymmetry derived from quantum field theory, defining an interesting subclass of systems. For such systems, the correspondence with bow diagrams remains unclear, but it is quite common to encounter situations where a system is given but the endpoints of the D3-branes are unknown. Our algorithm can also be used to determine supersymmetry in types B, C, and D. Indeed, it is known that, mirroring the system with respect to the orientifold reduces the problem to the type A case, either finite or affine.\footnote{We thank Guhesh Kumaran and Chunhao Li for bringing these additional cases to our attention.}
\end{Remark}

\subsection{Constructing supersymmetric brane diagrams}\label{s.constructingdiagrams}
Let us consider a bow diagram $(B,\Lambda,\vv)$ and assume that the previous algorithm has been successfully applied, thus proving that it is supersymmetric. At the conclusion of this process, we obtain a supersymmetric bow diagram of finite type A $(B,\Lambda'',\vv'')$. 
 
By proceeding as in the proof of Proposition \ref{Prop-finitetypeAsupersymmetry}, we can associate to this bow diagram a corresponding supersymmetric brane diagram. Specifically, we use Hanany--Witten transitions to transform $(B,\Lambda'',\vv'')$ in a separated diagram $(B,\Lambda''',\vv''')$ as in Figure \ref{StypeA}. Then, we construct a corresponding supersymmetric brane diagram in the following way. We attach the maximum possible number of fixed D3-branes to the NS5-brane corresponding to $\bigcirc_1$ without breaking supersymmetry, ensuring that the resulting brane diagram corresponds to a bow subdiagram of $(B,\Lambda''',\vv''')$. That is, we attach one D3-brane between $\bigcirc_1$ and $x_i$ for every $i\leq f = \operatorname{max}\Bigl\{ k \in \{0,1,\dots,w\} \ \mid \ v_{-j} \geq k-j \ \ \ \ \ \forall j\in\{0,\dots,k\}\Bigr\}$. If $f=0$, then $v_0=0$ and then every brane diagram corresponding to our bow diagram is supersymmetric. We then repeat the same procedure for $\bigcirc_2$ and continue sequentially. See Example \ref{ex2} for a detailed application of this procedure to separated bow diagrams of finite type A with two arrows and two x-points.

Using Hanany--Witten transitions on brane diagrams, we obtain a supersymmetric brane diagram for $(B,\Lambda'',\vv'')$.

Finally, let us consider the reversed algorithm in subsection \ref{algorithmsupersymmetry1} to go from $(B,\Lambda'',\vv'')$ to $(B,\Lambda,\vv)$. Each step in this reversal consists of a sequence of either a Hanany--Witten transitions or a supersymmetric increment. Since the interpretation of Hanany--Witten transitions for supersymmetric brane diagrams (see \S\ref{s.hw}) and the interpretation of supersymmetric increments for brane diagrams (Definition \ref{Def-susyincrements}), we can go from the supersymmetric brane diagram associated with $(B,\Lambda'',\vv'')$ to a supersymmetric brane diagram associated with $(B,\Lambda,\vv)$.

\subsection{Constructing points of bow varieties}
Let $(B, \Lambda, \vv)$ be a supersymmetric bow diagram. We now construct a solution for $\mu_{(B,\Lambda, \vv)} = 0$ by reversing the steps of the algorithm to detect supersymmetry (see \S\ref{algorithmsupersymmetry1}).

First, consider the supersymmetric finite type A bow diagram $(B,\Lambda'',\vv'')$ obtained at the end of the algorithm (see \S\ref{s.constructingdiagrams}). By Theorem \ref{Theorem-susyfinitetypeAimpliesnonemptiness} we can construct an explicit solution to the corresponding moment map equation. Let us briefly summarize how that theorem works. Let $\mathcal{B}$ be the supersymmetric brane diagram associated with $(B,\Lambda'',\vv'')$. If we remove all the unfixed D3-branes, we have a subdiagram which is Hanany--Witten equivalent to a brane diagram without D3-branes. That is, $(B,\Lambda'',\vv'')$ can be obtained from a bow diagram $(B,\Tilde{\Lambda},0)$ via supersymmetric increments and Hanany--Witten transitions. Since the moment map equation $\mu=0$ associated with $(B,\Tilde{\Lambda},0)$ has a solution, applying (the proofs of) Theorem \ref{supersymmetricincrementsforbows} and Theorem \ref{HWtransitionsforbow} (\cite[Theorem 7.1]{NT17}), we can extend such a solution to a solution to the moment map equation $\mu=0$ for $(B,\Lambda'',\vv'')$.

For what noticed in subsection \ref{s.constructingdiagrams}; the same argument can be applied to extend this solution to a solution for the moment map equation $\mu_{(B,\Lambda, v)} = 0$.

\begin{Remark}
In the proof of Theorem \ref{supersymmetricincrementsforbows}, we outline a concrete procedure for extending solutions. However, when dealing with supersymmetric increments between x-points, this extension is not necessarily unique.
\end{Remark}

%% file: 005example.tex
\section{Examples}\label{section5}
In this section, we present some examples illustrating the algorithm in action.
    \begin{Example}\label{ex1} Let us denote by $(B,\Lambda,\vv)$ the bow diagram:
    \begin{center}\input{Fig/notsusyexample2}
    \end{center}
From Proposition \ref{Prop-finitetypeAsupersymmetry} and Theorem \ref{maintheorem}, it is supersymmetric if and only if $\cD^1_{1,1}\geq0$ and $\cD^1_{1,2}\geq0$. However, $\cD^1_{1,2}=2+0+0-3=-1<0$ and indeed it is not supersymmetric.\end{Example}

\begin{Example}\label{ex2}
Consider the bow diagram $(B,\Lambda,\vv)$:  
\begin{center}
    \input{Fig/finitesusyexample}
\end{center}  
Since it is separated of finite type A, supersymmetry reduces to verifying the following inequalities:  
\begin{equation*}
    \begin{aligned}
        \cD^1_{1,1}=&1+v_1+v_{-1}-v_0\geq0, & \quad
        \cD^1_{1,2}=&2+v_1-v_0 \geq 0, \\
        \cD^1_{2,1}=&2+v_{-1}-v_0\geq 0, & \quad
        \cD^1_{2,2}=&4-v_0\geq0.\\
    \end{aligned}
\end{equation*}  
along with $v_{-1}, v_0, v_1 \geq 0$.  

    If these hold, as notice in \S\ref{s.constructingdiagrams}, Proposition \ref{Prop-finitetypeAsupersymmetry} provides a method to construct an associated supersymmetric brane diagram. Let us show how that algorithm works. To simplify the graphical representation of brane diagrams, when two 5-branes are connected by $m \geq 0$ D3-branes, we draw a single D3-brane labelled $m$, omitting labels when $m=1$.  

Since $v_0$ represents the number of fixed D3-branes, the inequalities ensure $v_0 \leq 4$. The first step is to compute  
\begin{equation*}
\begin{aligned}
  f &= \max\left\{k\in\{0,1,2\} \mid v_0 \geq k,\ v_{-1} \geq k-1,\ 0 \geq k-2\right\} \\
    &= \min\bigl(2, v_0, v_{-1}+1\bigr).
\end{aligned}
\end{equation*} 
which gives the maximum number of fixed D3-branes that can (and will) be attached to the NS5-brane corresponding to $\bigcirc_1$. As shown in Proposition \ref{Prop-finitetypeAsupersymmetry}, the inequalities always guarantee $v_1 \geq v_0 - f$.  
  
If $f=0$ then $v_0=0$ and there is only one possible brane diagram:  
  \begin{center}
      \input{Fig/bdiagramf=0}
  \end{center}  
If $f=1$, then $v_0 \in \{1,2\}$ and the corresponding brane diagram is given by:
  \begin{center}
      \input{Fig/bdiagramf=1}
  \end{center}  
Note, in this case if $v_0=2$ then $v_{-1}=0$.

If $f=2$, then $v_0 \in \{2,3,4\}$ and $v_{-1}\geq1$. In this case the corresponding brane diagram is given by:  
  \begin{center}
      \input{Fig/bdiagramf=2}
  \end{center}  

Finally, in the proof of Theorem \ref{Theorem-susyfinitetypeAimpliesnonemptiness}, we provided a systematic method to construct a solution to $\mu = 0$ from a supersymmetric bow diagram of finite type A. We now apply this procedure to our example, analysing each case according to the value of $f$.

If $f=0$ then we can think of $(B,\Lambda,\vv)$ as obtained from $(B,\Lambda,0)$ after applying supersymmetric increments (see Theorem \ref{supersymmetricincrementsforbows}). Therefore, a solution is given by:
\begin{center}
   \input{Fig/solutionf=0}
\end{center}
where $c_1,\dots,c_{v_{-1}}\neq0$ and $c_i-c_j \neq 0$ for $i\neq j$. We need to check condition (\ref{condition-S2}) for the first triangle and condition (\ref{condition-S1}) for the second triangle. Let us verify only condition (\ref{condition-S2}) for the first triangle since the other follows similarly. We have to check that the smallest $B^+_1$-invariant vector subspace of $\CC^{v_{-1}}$ containing $(1,\dots,1)$ is the whole space. In this case it reduces to observe that for every $k\in\{2,\dots,v_{-1}\}$,
\begin{equation*}
    (B_1^+-c_k)\cdots(B_1^+-c_{v_{-1}})\begin{bmatrix}
        1\\
        \vdots\\
        1
    \end{bmatrix}=\begin{bmatrix}
        (c_1-c_k)\cdots(c_1-c_{v_{-1}})c_1\\
        \vdots\\
        (c_{k-1}-c_k)\cdots(c_{k-1}-c_{v_{-1}})c_{k-1}\\
        0\\
        \vdots\\
        0
    \end{bmatrix}.
\end{equation*}

If $f=1$, we can consider the bow subdiagram $(B,\Lambda,\Tilde{\vv})$ corresponding to the brane diagram obtained from $\mathcal{B}$ by removing unfixed D3-branes, that is:
\begin{center}
    \input{Fig/reducef=1}
\end{center}
Although in this case it is not hard to construct a solution for the moment map, let us proceed as in the algorithm. We use Hanany--Witten transitions to move $x_1$ through the NS5-branes connected to it by a D3-branes. We obtain a bow diagram having dimension vector $0$ which admits one solution for its moment map equation $\mu=0$. Hence, also $(B,\Lambda,\Tilde{\vv})$ admits a solution $\Tilde{m}=(\Tilde{A},\Tilde{B}^\pm,\Tilde{a},\Tilde{b},\Tilde{C},\Tilde{D)}$ for the corresponding moment map equation. The solution $\Tilde{m}$ could be explicitly constructed by following the proof of \cite[Proposition 7.1]{NT17} and it is a collection of maps of the type:
\begin{center}
    \input{Fig/solutionreducedf=1}
\end{center}
where $\Tilde{D}_1\Tilde{C}_1=0$ and the kernel of $\Tilde{b}_1$ is not in the kernel of $\Tilde{B}_1^-$ (since $v_0\leq2$).

Finally, since we can go from the new bow diagram to the original one by applying supersymmetric increments, we can construct the desired solution from $\Tilde{m}$ as follows:
\begin{center}
\input{Fig/solutionf=1}
\end{center} 
where:
\begin{gather*}
        C_1=\begin{bmatrix}
           \Tilde{C}_1 & 0 
        \end{bmatrix}, \ \ D_1=\begin{bmatrix}
           \Tilde{D}_1\\
           0
        \end{bmatrix}, \ \ 
     b_1=\Tilde{b}_1,\ \ a_1=\begin{bmatrix}
         1\\
         \vdots\\
         1
     \end{bmatrix},\ \ b_2=\begin{bmatrix}
         1&\dots&1
     \end{bmatrix} \\
     B_1^-=\Tilde{B}_1^-, \ \  B_1^+=B_2^-=\begin{bmatrix}
         c_1 & 0 & 0 \\
         0 & \ddots & 0 \\
         0 & 0 & c_{v_{-1}}
     \end{bmatrix},\ \
     A_1=\begin{bmatrix}
         b_1(B_1^{-}-c_1)^{-1}\\
         b_1(B_1^{-}-c_2)^{-1}\\
         \vdots\\
         b_1(B_1^{-}-c_n)^{-1}
     \end{bmatrix}
    \end{gather*}
where $c_1,\dots,c_n \neq 0$ such that $B_1^--c_i$ is invertible for every $i$ and $c_i-c_j\neq 0$ for every $i\neq j$.

The case $f=2$ is analogous to the previous ones, we leave the details to the reader.
\end{Example}

\begin{Example}
    Let us denote by $(B,\Lambda,\vv)$ the bow diagram with one arrow and one x-point:
    \begin{center}
\input{Fig/11bowdiagram}
\end{center}
Let us apply the algorithm to detect supersymmetry in this case (see \S\ref{algorithmsupersymmetry1}). It is separated and therefore the first step of the algorithm is trivial. Without loss of generality, $v_0-v_{1}\geq0$. Then, the second step of the algorithm consists in applying $(v_0-v_{1})$-times the clockwise Hanany--Witten transition to move $x_1$ through $\bigcirc_1$. To apply this step we have to check the inequalities:
\begin{equation*}
    \cD_{1,1}^t=1+2t+\frac{(t-1)(t-2)}{2}+(t+1)v_1-tv_0\geq0, \ \ t\in\{1,\dots,v_0-v_1\}
    \end{equation*}
If all of them are satisfied, then, by applying such transitions, we obtain the diagram:
\begin{equation}
    \begin{aligned}\input{Fig/11bowdiagramreduced}\label{1-1}\end{aligned}
\end{equation}
where $\cD^t_{1,1}$ extended to $t=0,-1$ behaves as expected, that is: $\cD^{-1}_{1,1}=v_0$ and $\cD^{0}_{1,1}=v_1$. Since $\cD_{1,1}^{v_0-v_{1}-1}-\cD_{1,1}^{v_0-v_{1}}=0$, it is clearly supersymmetric and, for instance, an associate supersymmetric brane diagram is given by stretching $\cD_{1,1}^{v_0-v_1}$ unfixed D3-branes doing one loop.

A brane diagram for the original bow diagram is obtained by applying $(v_0-v_1)$ Hanany--Witten transitions moving $x_1$ anticlockwise through $\bigcirc$ to the brane diagram associated with (\ref{1-1}). The brane diagram giving rise to the original bow diagram is obtained by adding for every $p\in\{0,\dots,v_0-v_1-1\}$ a anticlockwise-oriented fixed D3-brane with winding number $p$.

The solution to $\mu_{(B,\Lambda,\vv)}=0$ constructed by the algorithm involves two steps. First, the proof of Theorem \ref{supersymmetricincrementsforbows}, gives us the solution $(A=\operatorname{id},B^\pm=0,a=0,b=0,C=0,D=0)$ to $\mu=0$ where $\mu$ is the moment map associated with the bow diagram (\ref{1-1}). Then, applying the Hanany--Witten transitions we mentioned above, one can explicitly construct a solution for $\mu_{\Lambda,\vv}=0$.
\end{Example}

\begin{Example}
       We denote by $(B,\Lambda,\vv)$ the bow diagram:
    \begin{center}
\input{Fig/22bowdiagram}
\end{center}
Let us see how the algorithm for supersymmetry works in this case. It is separated and therefore the first step is trivial. We can assume $v_0-v_2\geq0$. Let $d$ be the greatest integer less than or equal to $(v_0-v_2)/2$. The second step consists in checking supersymmetry inequalities corresponding with moving $x_1, x_2$ anticlockwise through all the arrows $d$-times. Therefore, we have to check $2d$ inequalities. Suppose they are all satisfied, then, after performing the corresponding Hanany--Witten transitions, we reduce to the case $0\leq v_0-v_2\leq1$.

The next step in the algorithm consists of studying supersymmetry for a bow subdiagram obtained by decreasing $v_0,v_1,v_2$. However, the same argument can be used to reduce the problem to study supersymmetry of:
\begin{center}
    \input{Fig/22bowdiagram2}
\end{center}
where $a\leq v_1$, $b\leq v_{-1}$ maximising the sum $a+b \leq v_2$. 

If $a+b=v_2$, it is supersymmetric since $0\leq v_0-v_2\leq1$. Otherwise, we reduced the problem in studying:
\begin{center}
    \input{Fig/22bowdiagram3}
\end{center}
for which supersymmetry is equivalent to $v_0-(v_1+v_{-1}),\ v_2-(v_1+v_{-1})\in\{0,1\}$. However, we have $v_0\geq v_2 > v_1+v_{-1}$ and therefore, under this assumption, it is supersymmetric if and only if $v_0=v_2=v_1+v_{-1}+1$.

The construction of the supersymmetric brane diagram and the solution to $\mu = 0$ is a mix of the previous examples.
\end{Example}

%% file: 00Appendix2.tex
\appendix
\section{Proof of $(c)\Longleftarrow(d)$}\label{InverseStep2-Appendic}
In this section we provide a direct proof of the reverse implication in $Step\ 2$ (see \S\ref{s.step2}). That is, let us consider a separated bow diagram $(B,\Lambda,\vv)$ as in Figure \ref{StypeAffine}. We prove that if it satisfies supersymmetry inequalities (\ref{susyinequalities}) then every Hanany--Witten equivalent bow diagram originates from a brane diagram (i.e. it has nonnegative dimensions). As noticed at the beginning of subsection \ref{s.step2}, we can assume $n,w>0$.

\begin{Definition} Consider a bow diagram $(B,\Lambda,\vv)$. Let us fix an x-point $x$ and an arrow $e$. Given a sequence of Hanany--Witten transitions, we define $C(x,e)$ to be the total number of transitions moving $x$ clockwise through $e$, minus the total number of of transitions moving $x$ anticlockwise through $e$.
\end{Definition}
\begin{Lemma}\label{Lemma1-reversestep2}
    Consider a separated bow diagram $(B,\Lambda,\vv)$ as in Figure \ref{StypeAffine} and suppose to apply a sequence of Hanany--Witten transitions. Then $C(x_1,e_n)\geq0$ or $C(x_w,e_1)\leq0$.
\end{Lemma}
\begin{proof}
    It follows by noticing that, to move $x_1$ anticlockwise through $e_n$, we need to move $x_w$ anticlockwise through $e_n$. Similarly, to move $x_w$ clockwise through $e_1$, we need to move $x_1$ clockwise through $e_1$.
\end{proof}
\begin{Lemma}\label{Lemma2-reversestep2}
    Consider a separated bow diagram $(B,\Lambda,\vv)$ as in Figure \ref{StypeAffine}. Suppose to apply a sequence of Hanany--Witten transitions and that there exist $i\in\{1,\dots,w\}$ and $s\in\{1,\dots,n\}$ such that $C(x_i,e_s)=0$. Then:
    \begin{equation*}
        \begin{aligned}
            2\geq& C(x_k,e_d)\geq0 & & \text{if }k<i,\ d<s\\
            0\geq& C(x_k,e_d)\geq-2 & & \text{if }k>i,\ d>s\\
            1\geq& C(x_k,e_d)\geq-1 & & \text{if }k\leq i, d\geq s \text{ or }k\geq i, d\leq s
        \end{aligned}
    \end{equation*}
\end{Lemma}
\begin{proof}
We prove only that $2\geq C(x_k,e_d)\geq0$ if $k<i$ and $d<s$, as the other inequalities follow by completely analogous arguments. Given $k<i$ and $d<s$, if we move $x_k$ through $e_d$ clockwise $t\geq2$ times, then, since we cannot move an x-point through another x-point, we need to move $x_i$ through $e_s$ clockwise $t-2$ times. Hence, $2\geq C(x_k,e_d)$. Similarly, if we move $x_k$ through $e_d$ anticlockwise $t\geq1$ times, then we need to move $x_i$ through $e_s$ $t\geq 1$ times. Hence, $C(x_k,e_d)\geq0$.
\end{proof}
From Lemma \ref{Lemma-interpretationsusyinequalities}, the claim follows from the following proposition.
\begin{Proposition}
    Consider a separated bow diagram $(B,\Lambda,\vv)$ as in Figure \ref{StypeAffine}. Each entry of the dimension vector of every Hanany--Witten equivalent bow diagram is given by $\cD^t_{s,k}$ or $\aD^t_{s,k}$ for some $t\in\NN^+$, $s \in \{0,1,\dots,n\}$ and $k\in\{0,1,\dots,w\}$.
\end{Proposition}
\begin{proof} If either $n=0$ or $w=0$ then the claim is trivial (see \S\ref{s.step2}).

Let us assume $n,w>0$. Let us notice that, since every Hanany--Witten transition has an inverse (see (\ref{HWprocess})), any two Hanany--Witten equivalent bow diagrams are connected by infinitely many sequences of Hanany--Witten transitions.

Let us apply a sequence of Hanany--Witten transitions to $(B, \Lambda, \vv)$ and denote by $(B, \Lambda', \vv')$ the resulting bow diagram.
By Lemma \ref{Lemma1-reversestep2}, $C(x_1,e_n)\geq0$ or $C(x_w,e_1)\leq0$.

Suppose $t\coloneqq C(x_1,e_n)\geq0$. That is, the effect of the considered sequence on $x_1$ consists of moving it through $\bigcirc_n$ in the clockwise direction exactly $t$ times, with $t \geq 0$. Then, we can obtain $(B,\Lambda',\vv')$ by applying two sequence of transitions as follows. First, we move all the x-points clockwise through all the arrows $t$ times. Thanks to Lemma \ref{Lemma-interpretationsusyinequalities}, this sequence of transitions generates the bow diagram:
    \begin{equation}\label{midstep-appendix}
        \input{Fig/propstep2}
    \end{equation}
To obtain $(B,\Lambda',\vv')$, we left to apply a sequence of transitions to (\ref{midstep-appendix}) for which $C(x_1,e_n)=0$. From Lemma \ref{Lemma2-reversestep2}, for such sequence of Hanany--Witten transitions we have $1\geq C(x_k,e_d)\geq-1$ for every $k\in\{1,\dots,w\}$ and $s\in\{1,\dots,n\}$. In particular, we can choice a sequence of transitions such that every x-point goes through an arrow at most once and $x_1$ never cross $e_n$. 

Let us fix such a sequence. Then there exists a partition of the set of x-points in three (possibly empty) disjoint subsets: $\{x_1,\dots,x_k\}$, $\{x_{k+1},\dots,x_{d-1}\}$ and $\{x_d,\dots,x_w\}$ such that: we move $x_1,\dots,x_k$ clockwise without crossing $e_n$, we do not move $x_{k+1},\dots,x_{d-1}$ and we move $x_d,x_{d+1},\dots,x_w$ anticlockwise. 

Let us notice that the set of segments involved in the transitions of $x_1,\dots,x_k$ and the set of segments involved in the transitions of $x_d,\dots,x_w$ are disjoint.

Then, the entries of $\vv'$ corresponding to the segments in the arc starting from $x_w$ towards $x_1$ in the anticlockwise direction and in the arc starting from $e_1$ towards $e_n$ in the anticlockwise direction are given by $\cD^{t}_{s,k}$ for some $s,k$. If $\{x_1,\dots,x_k\}$ is non-empty, the dimensions on the segments in the arc starting from $x_1$ towards $e_1$ in the anticlockwise direction are given by $\cD^{t+1}_{s,k}$ for some $s,k$. 

Finally, we claim that if $\{x_d,\dots,x_w\}$ is non-empty, the dimensions on the segments in the arc starting from $e_n$ towards $x_w$ in the anticlockwise direction are given by $\cD^{t}_{s,k}$ for some $s,k$. If $t>0$ this is clear. If $t=0$, thanks to Remark \ref{Inequalitiesfort=0}, these dimensions are given by $\aD^1_{s,k}$ for some $s,k$, as expected. 

Analogous considerations hold if we consider the case $C(x_w,e_1)\leq0$.
\end{proof}

%% file: 00Appendix.tex
\section{Weights of affine Lie algebras and level-rank duality}\label{s.weights}
Here we fix notations and recall useful facts about weights of affine Lie algebras, their relation with generalized Young diagrams and the induced level-rank duality.
\subsection{Weights of finite Lie algebras}\label{s.A1}
Let us fix notations for weights of $\Liesl_w$ and $\Liegl_w$. We fix the Cartan subalgebra of $\Liegl_w$ (resp. $\Liesl_w$) given by the diagonal matrices (resp. traceless diagonal matrices). Let us denote by $\epsilon_1,\dots,\epsilon_w$ the canonical basis of the dual of the Cartan subalgebra and identify the space of weights of $\Liegl_w$ (resp. $\Liesl_w$) with $\ZZ^w$ (resp. $\ZZ^w/\ZZ(1,\dots,1)$).

We denote by $\alpha_1, \dots, \alpha_{w-1}$ the simple roots of $\mathfrak{sl}_w$, and by $\Lambda_1, \dots, \Lambda_{w-1}$ the fundamental weights of $\mathfrak{sl}_w$. In coordinates:
\begin{equation*}
    \alpha_i=[0,\dots,0,\underset{i}{1},\underset{i+1}{-1},0,\dots,0], \ \ \Lambda_i=[\underbrace{1,\dots,1}_{i},0,\dots,0] \in \ZZ^w,
\end{equation*}
for $i=1,\dots,w-1$.
For a later purpose, let us define $\Lambda_w=[1,\dots,1]\in\ZZ^w$, that is the weight corresponding to the trace.

There is a natural projection from $\mathfrak{gl}_w$-weights to $\mathfrak{sl}_w$-weights given by factoring out the trace map, that is: $\mathbb{Z}^w \to \mathbb{Z}^w / \mathbb{Z}(1,\dots,1)$.
Given a $\Liegl_w$-weight $\lambda=\sum\limits_{i=1}^{w}\lambda_i\epsilon_i$, the corresponding $\Liesl_w$-weight is given by:
\begin{equation*}
    \sum\limits_{i=1}^{w-1}(\lambda_i-\lambda_{i+1})\Lambda_i=\sum\limits_{i=1}^{w-1}\bigl(\sum\limits_{j=1}^{i}\lambda_j\bigr)\alpha_i-\vert\lambda\vert\Lambda_{w-1}.
\end{equation*}
where $\vert\lambda\vert=\langle\lambda,\operatorname{diag}(1,\dots,1)\rangle=\sum\limits_{i=1}^{w}\lambda_i$ is known as the charge of $\lambda$.

\begin{Definition}
    Given $\lambda=[\lambda_1,\dots,\lambda_w]\in\ZZ^w$. It is dominant as $\Liegl_w$-weight and $\Liesl_w$-weight, if $\lambda_1\geq\cdots\geq\lambda_w$. It is a polynomial dominant $\Liegl_w$-weight if $\lambda_1\geq\cdots\geq\lambda_w\geq0$.
\end{Definition}

\subsection{Weights of affine Lie algebras}\label{s.A2}
Let $\widehat{\mathfrak{sl}}_w$ (resp. $\widehat{\Liegl_w}$) be the central extension of the loop algebra associated with $\mathfrak{sl} _w$ (resp. $\Liegl_w$), and let $(\mathfrak{sl}_w)_{\text{aff}}$ (resp. $(\Liegl_w)_{\text{aff}}$) be the corresponding Kac--Moody affine Lie algebra.

Let $\alpha_0, \alpha_1, \dots, \alpha_{w-1}$ denote the simple roots of $(\mathfrak{sl}_w)_{\text{aff}}$, $\delta = \alpha_0 + \alpha_1 + \cdots + \alpha_{w-1}$ be the null root and let $\Lambda_0, \Lambda_1, \dots, \Lambda_{w-1}$ be the fundamental weights of $(\Liesl_w)_{\text{aff}}$.
 
If $d$ is the degree operator in the Cartan subalgebra of $(\mathfrak{sl}_w)_{\text{aff}}$, our convention is 
 \begin{equation*}
     \langle \Lambda_i,d\rangle=0, \ \ \ \langle \alpha_i,d \rangle=\delta_{i0}, \ \ \ \text{for }i=0,\dots,w-1.
 \end{equation*}

Simple roots and fundamental weights are related by:
\begin{equation}\label{eq.simpleandfundamental}
    \alpha_i=2\Lambda_i-\Lambda_{i-1}-\Lambda_{i+1}+\delta_{0i}\delta,
\end{equation}
where $i\in\ZZ/\ZZ w$.

The notion of charge of a $\Liegl_w$-weight extends to weights of $\widehat{\Liegl_w}$ and $\bigl(\Liegl_w\bigr)_{\text{aff}}$. That is, if $\lambda=x\Lambda_0+\sum\limits_{i=1}^{w}\lambda_1\epsilon_i+y\delta$, the charge of $\lambda$ is $\vert\lambda\vert=\sum\limits_{i=1}^{w}\lambda_i$.

If $c$ is the generator of the central extension of $(\widehat{\Liesl}_w)$, then the \textit{level} of a weight $\gamma$ is defined as $\langle \gamma,c \rangle$. Note, if $\gamma=\sum\limits_{i=0}^{w-1}\gamma_i\Lambda_i + m\delta$, then the level of $\gamma$ is $\sum\limits_{i=0}^{w-1}\gamma_i$.

\begin{Remark}\label{projectionmodulodelta}
The previous lattices are related by:
\begin{equation}\label{projectionsweights}
\begin{tikzcd}
(\Liegl_w)_{\text{aff}}\text{-weights} \arrow[r] \arrow[d] & \widehat{\Liegl}_w\text{-weights} \arrow[d] \\
(\Liesl_w)_{\text{aff}}\text{-weights} \arrow[r] & \widehat{\Liesl}_w\text{-weights}
\end{tikzcd}\end{equation}
where the horizontal maps are given by the quotient modulo $\ZZ\delta$ and the vertical maps are given by the quotient modulo $\ZZ\Lambda_w$.
\end{Remark}

\begin{Definition}\label{definitiondominanceorder}
The set of dominant $(\Liesl_w)_{\text{aff}}$-weights is defined as $\bigoplus\limits_{i+0}^{w-1}\NN\Lambda_i\oplus\ZZ\delta$. If $\lambda,\mu$ are $(\Liesl_w)_{\text{aff}}$-weights, we say that $\lambda\geq\mu$ in the dominance order if $\lambda-\mu$ is sum of simple roots.

The same definition extends to the lattice of $(\Liegl_w)_{\text{aff}}$-weights via (\ref{projectionsweights}).
\end{Definition}

\subsection{Generalized Young diagrams and level-rank duality}\label{s.A3}
In order to state the level-rank duality for weights of affine Lie algebras, it is useful to recall the notion of generalized Young diagram. See \cite{nakanishi1992level}, \cite[Appendix A]{nak09}, \cite[Section 3(v) and 3(vi)]{nakajima2018towards} and \cite[Section 7.6]{NT17}.

A \textit{generalized Young diagram with $w$ rows and level $n$ constraint} is an integral vector $[\lambda_1, \lambda_2, \dots, \lambda_w] \in \ZZ^w$ satisfying the following conditions:
\begin{equation*}
        \lambda_1 \geq \lambda_2 \geq \dots \geq \lambda_w\geq\lambda_1-n.
\end{equation*}
We denote the set of all such generalized Young diagrams by $\Y_w^n$. Using Maya diagrams, we can graphically represent generalized Young diagrams as follows. Given $[\lambda_1, \dots, \lambda_w] \in \Y_w^n$, we construct infinitely many blocks labelled with $N \in \frac{1}{2} + \ZZ$. Each block is decomposed into $w \times n$ boxes, where $w$ is the number of rows and $n$ is the number of columns. For $1 \leq i \leq w$ and $1 \leq s \leq n$, we denote by $m_N(i, s)$ the box $(i, s)$ in the block $N$. If $n(N - \frac{1}{2}) + s \leq \lambda_i$, we colour the box $m_N(i, s)$ grey; otherwise, we leave it white.

For instance, if $w=2$ and $n=3$, then the Maya diagram corresponding to the generalized Young diagram $\lambda=[4,1]\in\Y_2^3$ is given by:
\begin{equation*}
    \input{Fig/Mayadiagram1}
\end{equation*}
By definition, the Maya diagram associated with a generalized Young diagram \textit{sits} in two blocks. That is, there exists $N\in\frac{1}{2}+\ZZ$ such that for every $i,s$, the box $m_{N'}(i,s)$ is gray if $N'<N$ and it is white if $N'>N+\frac{1}{2}$. Also, a generalized Young diagram is a usual Young diagram if and only if its Maya diagram sits in the block $N=\frac{1}{2}$, that is, for every $i,s$, the box $m_{N'}(i,s)$ is gray for all $N'<\frac{1}{2}$ and it is white for all $N'>\frac{1}{2}$.

Thanks to Maya diagrams, we can define the transpose of $[\lambda_i]\in\Y_w^n$ as the generalized Young diagram associated with the Maya diagram obtained by transposing every block of the original Maya diagram along its diagonal. We denote the transpose generalized Young diagram (with $n$ rows and level $w$ constraint) by $[\prescript{t}{}{\lambda}_s]\in\Y_n^w$.

For instance, with respect to the above example, the transposed Maya diagram is given by: 
\begin{equation*}
    \input{Fig/Mayadiagram2}
\end{equation*}
which corresponds to $\prescript{t}{}{\lambda}=[3,1,1]\in\Y_3^2$.

A $(\Liegl_w)_{\text{aff}}$-weight $\lambda$ bijectively corresponds to an integral vector $[\lambda_1,\dots,\lambda_w]$ together with an integer $l(\lambda)$ and an integer $\langle\lambda,d\rangle$ by:
\begin{equation*}
    \Bigl(l(\lambda)-(\lambda_1-\lambda_w)\Bigr)\Lambda_0+\sum\limits_{i=1}^{w-1}(\lambda_i-\lambda_{i+1})\Lambda_i+\lambda_w\Lambda_w+\langle\lambda,d\rangle\delta.
\end{equation*}
Thus, we have a bijection between dominant $(\Liegl_w)_{\text{aff}}$-weights of level $n$ and $\Y_w^n\times\ZZ$.
The projection of $(\Liegl_w)_{\text{aff}}$-weights to $\widehat{\Liegl}_w$-weights induces a bijection between $\Y_w^n$ and dominant $\widehat{\Liegl}_w$-weights of level $n$.
The projection of $\widehat{\Liegl}_w$-weights to $\widehat{\Liesl}_w$-weights induces a map from $\Y_w^n$ to dominant $\widehat{sl}_w$-weights of level $n$ given by the formula (see eq. (3.16) in \cite{nakanishi1992level}):
\begin{equation}\label{eq.diagramstoweights}
 \Bigl(n-(\lambda_1-\lambda_w)\Bigr)\Lambda_0+\sum\limits_{i=1}^{w-1}(\lambda_i-\lambda_{i+1})\Lambda_i.
\end{equation}

Therefore, the transposition of generalized Young diagrams gives a bijection between dominant $\widehat{\Liegl}_w$-weights of level $n$ and dominant $\widehat{\Liegl}_n$-weights of level $w$. 

\begin{Definition}\label{eq.transposeweights}
   Let $\lambda$ be a dominant $(\Liegl_w)_{\text{aff}}$-weight of level $n$, given by $[\lambda_i]\in\Y_w^n$ and $\langle\lambda,d\rangle\in\ZZ$. Then, we define its transposed $\prescript{t}{}{\lambda}$ as the dominant $(\Liegl_n)_{\text{aff}}$-weight of level $w$ given by $[\prescript{t}{}{\lambda}_i]\in\Y_w^n$ and $-\langle\lambda,d\rangle\in\ZZ$. 
\end{Definition}

\begin{Remark}\label{remarkdominanceorder} Thanks to (\ref{eq.simpleandfundamental}), we have:
\begin{equation}\label{eq.changingweightbasis}
\begin{aligned}
   \lambda=(n-\lambda_1+\lambda_w)\Lambda_0+\sum\limits_{i=1}^{w-1}(\lambda_i-\lambda_{i+1})\Lambda_i+\langle\lambda,d\rangle\delta\\=(n+|\lambda|)\Lambda_0+\sum\limits_{i=1}^{w-1}\Bigl(\sum\limits_{j=1}^{i}\lambda_j\Bigr)\alpha_i-|\lambda|\Lambda_{w-1}+\langle\lambda,d\rangle\delta
\end{aligned}\end{equation}
Let $\lambda,\mu$ be $(\Liesl_w)_{\text{aff}}$-weights (or $(\Liegl_w)_{\text{aff}}$-weights) of the same level and suppose there exist associated integral vectors, say $[\lambda_i],[\mu_i]\in\ZZ^w$, with the same charge. Thanks to (\ref{eq.changingweightbasis}), we have $\lambda\geq\mu$ in the dominance order if and only if $$\sum\limits_{i=1}^{j}\lambda_i-\sum\limits_{i=1}^{j}\mu_i+\langle\lambda-\mu,d\rangle\geq0$$ for every $j\in\{1,\dots,w\}$.
\end{Remark}